\documentclass[12pt]{article}
\usepackage{etex}
\usepackage{amsmath,amssymb,amsfonts,amsthm}
\usepackage{a4wide}
\usepackage{color}
\usepackage{verbatim}
\usepackage{epsfig,subfigure,epstopdf}
\usepackage[]{graphics}
\usepackage{graphicx,inputenc}
\graphicspath{{figures_V1031/}}%
\usepackage{subfigure}
\usepackage[justification=centering]{caption}
\usepackage{pictex}
\usepackage{wrapfig}
\usepackage{multirow}
\usepackage[T1]{fontenc}
\usepackage{authblk}
\usepackage[textwidth=6.0in,textheight=9in]{geometry}

\usepackage[colorlinks,linktocpage,linkcolor=blue]{hyperref}
\usepackage{fancyhdr}

\newtheoremstyle{exampstyle}
  {\topsep} % Space above
  {\topsep} % Space below
  {\itshape} % Body font
  {} % Indent amount
  {\bfseries} % Theorem head font
  {.} % Punctuation after theorem head
  {.5em} % Space after theorem head
  {} % Theorem head spec (can be left empty, meaning `normal')

\theoremstyle{exampstyle}

\newtheorem{theorem}{Theorem}
\newtheorem{lemma}{Lemma}
\newtheorem{assumption}{Assumption}
\newtheorem{remark}{Remark}
\newtheorem{definition}{Definition}

\newtheorem{proposition}{Proposition}

\usepackage{parskip}
\let\oldref\ref
\renewcommand{\ref}[1]{(\oldref{#1})}  % stupid kludge for laTeX
\renewcommand{\eqref}[1]{(\oldref{#1})}

\newbox\boxaddrone \newbox\boxaddrtwo

\def\D{\partial_t^\alpha}
\def\N+{n\in\mathbb{N}^{+}}

\def\A{\mathcal{A}}
\def\l{\langle}
\def\rd{\rangle_{L^2(D)}}

\def\K{\mathcal{K}}

\def\sp{\text{Span}}

\def\V{\mathbb{V}}
\def\E{\mathbb{E}}

\def\A{\mathcal{A}}

\def\I{I_t^\alpha}
\def\o{\omega}
\def\R{\mathbb{R}}
\def\Z{\mathbb{Z}}
\def\W{\mathbb{W}}
\def\Vc{\mathcal{V}}

\begin{document}

\title{\Large\textbf{Reconstruction of the time-dependent source term in a  
stochastic fractional diffusion equation}}
\author[1]{Chan Liu\thanks{chanliu16@fudan.edu.cn}}
\author[2]{Jin Wen\thanks{wenj@nwnu.edu.cn}}
\author[3]{Zhidong Zhang\thanks{zhidong.zhang@helsinki.fi}}
\affil[1]{\normalsize{School of Mathematical Sciences, Fudan University, China}}
\affil[2]{\normalsize{Department of Mathematics, Northwest Normal University, China}}
\affil[3]{\normalsize{Department of Mathematics and Statistics, University of Helsinki, Finland}}

\maketitle

\begin{abstract}
In this work, an inverse problem in the fractional diffusion equation  
with random source is considered. Statistical 
moments are used of the realizations of single point observation 
$u(x_0,t,\omega).$ We build the representation of the solution $u$ in 
integral sense, then prove some theoretical results as uniqueness and 
stability. After that, we establish a numerical algorithm to solve the 
unknowns, where the mollification method is used. \\\\
\text{Keywords}: inverse problem, stochastic fractional diffusion equation, 
 random source, Volterra integral equation, mollification.
\end{abstract}

\section{Introduction}
\subsection{Mathematical statement} \label{sec:1.1}
In this work, the following stochastic fractional diffusion equation is 
considered, 
 \begin{equation}\label{SDE}
  \begin{cases}
   \begin{aligned}
    \D u+\A u&=f(x)[g_1(t)+g_2(t)\dot{\W}(t)], 
    &&(x,t)\in D\times(0,T],\ \alpha\in(1/2,1),\\
    u(x,t)&=0, &&(x,t)\in \partial D\times(0,T],\\
    u(x,0)&=0, &&x\in D,
   \end{aligned}
  \end{cases}
 \end{equation}
in which $D\subset\mathbb{R}^d,\ 1\le d\le 3$ is bounded with 
sufficiently smooth boundary and 
$\D$ means the Djrbashyan-Caputo fractional derivative, 
\begin{equation*}
	\D \varphi = \frac{1}{\Gamma(1-\alpha)}
	\int_0^t(t-\tau)^{-\alpha}\varphi'(\tau)\,d\tau.
\end{equation*}
For the source term $f(x)g(t,\omega):=f(x)[g_1(t)+g_2(t)\dot{\W}(t)]$, 
the spatial component $f(x)$ 
is deterministic and known, and the unknown $g(t,\omega)$ contains uncertainty, 
where $\o$ is the random variable. 
In this paper, we consider the case that $g(t,\omega)=g_1(t)+g_2(t)\dot{\W}(t)$ 
be an Ito process, and the notation $\dot{\W}(t)$ means the white noise derived from 
the time dependent Brownian motion. For the derivation and properties of 
Ito process, see \cite{Bernt2003stochastic} for details. 
For the two unknowns $g_1,g_2$ in $g(t,\o)$, we try to use the 
statistical moments from the observation $u(x_0,t,\o)$  
to recover them. However, from the properties of $\W$, 
which are displayed in section \ref{section:preliminary}, 
the sign of $g_2$ can not effect the stochastic process $g_2(t)\dot{\W}(t)$. Sequentially, 
$g_2^2$ is concerned instead of $g_2$.

The precise mathematical description of this inverse problem 
is given as follows, using the 
statistical moments of the realizations of 
$h(t,\o):=u(x_0,t,\o),\ x_0\in D$ to recover $g_1$ and $g_2^2$ 
simultaneously.

\subsection{Physical background} 
In microscopic level, the random motion of a single particle can be 
regarded as a diffusion process. Under the assumption that 
the mean squared displacement of jumps after a long time 
is proportional to time, i.e. 
$\overline{(\Delta x)^2}\propto t,\ t\to \infty$, 
the classical diffusion equation to describe the motion of particles 
can be derived. However, recently, some anomalous diffusion phenomena have been 
found, accompanying with considerable physical evidence, 
\cite{KlafterSilbey:1980, BarkaiMetzlerKlafter:2000,MagdziarzWeronBurneckiKlafter:2009}. 
In such \emph{anomalous diffusions}, the assumption 
$\overline{(\Delta x)^2}\propto t,\ t\to \infty$ may be 
violated, and it may possess the asymptotic behavior of $t^\alpha$, i.e. 
$\overline{(\Delta x)^2}\propto t^\alpha,\ \alpha\ne1.$ 
The different rate leads to a reformulation of the diffusion equation, 
introducing the time fractional derivative in it, and the corresponding 
equations are called fractional differential equations (FDEs). 
We list some of the applications of FDEs, 
to name a few, the diffusion process in a medium with fractal geometry 
\cite{nigmatullin1986realization},
non-Fickian transport in geological formations \cite{berkowitz2006modeling}, 
mathematical finance \cite{GorenfloMainardiScalasRaberto:2001}, 
theory of viscoelasticity \cite{bagley1983theoretical,Koeller:1984} 
and so on.

If uncertainty is added in the right-hand side, then this FDE system 
will become more complicated and interesting. This means that the random property in the source 
term will lead to a solution expressed as a stochastic process.
Considering this case is meaningful since it is 
common to meet a diffusion source defined as a stochastic process
 to describe the uncertain character imposed by nature. 
Hence, mathematically, it is worth to investigate the diffusion system 
with a random source.

\subsection{Previous literature} 
Considerable researchers have made efforts in the investigation of 
fractional differential equations and some valuable work are 
produced. For a comprehensive understanding of fractional calculus and 
fractional differential equations, we refer to
\cite{KilbasSrivastavaTrujillo:2006,SamkoKilbasMarichev:1993, 
baleanu2019handbook} and the references therein;
for numerical approaches, see 
\cite{ShenLiuAnh:2019, LvAzaiezXu:2018,HouHasanXu:2018,LvXu:2016,
ZengLiLiuTurner:2013,LiuLiuZeng:2019}.
%by Sakamoto and Yamamoto \cite{sakamoto2011initial} to study the initial 
%and boundary value problems for FDEs and the work by Luchko 
%\cite{Luchko2009maximum, Luchko2011maximum} to establish the maximum principle 
%in FDEs. Moreover, Jin, Lazarov and Zhou \cite{JinLazarovZhou:2016} 
%gave a numerical scheme to approximate the FDE by the finite element method. 

In the field of inverse problems, for an extensive review,  
\cite{JinRundell:2015} is referred. For time-fractional inverse 
problems, see \cite{LiChengLi:2019,KaltenbacherRundell:2019,BabaeiBanihashemi:2019,
QasemiRostamyAbdollahi:2019,TuanLuuTatar:2019,TranNguyenPhamMachNguyen:2019,
YanWei:2019,RuanZhangXiong:2018,HuangLiYamamoto:2019,ZhangZhou:2017,
LiuZhang:2017,Zhang:2017}; for fractional Calderon problems, 
see \cite{HarrachLin:2019,RulandSalo:2018,CaoLinLiu:2019,LaiLin:2019,
GhoshRulandSaloUhlmann:2018}. 
Furthermore, if we extend the assumption 
$\overline{(\Delta x)^2}\propto t^\alpha$ to a more general case 
$\overline{(\Delta x)^2}\propto F(t)$, the multi-term fractional 
diffusion equations and even the distributed-order differential equations 
will be generated, see \cite{RundellZhang:2017, 
LiLuchkoYamamoto:2017,ChengYuanLiang:2019} for details. 

The inverse problems of determining the uncertain unknowns 
also have drawn more and more attention from researchers. 
We refer to \cite{FengLiWang:2019,BaoXu:2013,Li:2011,BaoChenLi:2016,
LiYuan:2017,LiChenLi:2018,BaoChowLiZhou:2014,NiuHelinZhang:2018} 
and the references therein.

\subsection{Main results and outline}
Throughout this paper, the following restrictions on $f$ and  
the unknowns $g_1,g_2$ are assumed to be valid.

\begin{assumption}\label{assumption}
Let the function $f$ and the unknowns $g_1,g_2$ satisfy the following conditions.
 \begin{itemize}
  \item $g_1, g_2\in C[0,T],$ 
  \item $f\in \mathcal D(\mathcal{A}^2)\subset H^4(D)$ and 
  $f(x_0) \neq 0$.
  %\item $\E[ I^{1-\alpha}_t h(\cdot,\o)], \V[ I^{1-\alpha}_t h(\cdot,\o)]
  %\in L^2(0,T)$.
 \end{itemize}
\end{assumption}
The definition of the space $\mathcal D(\mathcal{A}^2)$ is stated in 
section \ref{section:preliminary}. 

Then we can state our main theorem, the stability theorem. 
It shows that $g_1,g_2^2$ can be bounded by the statistical moments of 
the realizations of single point data $u(x_0,t,\o)$. 
In this theorem, $g_2^2$ is used as we have mentioned in Section \ref{sec:1.1},
reflecting the linearity between $g_2^2$ and $\V[ I^{1-\alpha}_t h(t,\o)]$. 
Also, denote $|\cdot|_{H^{-1}(0,T)}$ as the seminorm of space 
$H^{-1}(0,T)$, $I_t^{1-\alpha}$ as the fractional integral 
operator, and $\E,\V$ as expectation and variance, 
respectively. Those knowledge can be found in section \ref{section:preliminary}.

\begin{theorem}[Stability]\label{stability}
 Under Assumption \ref{assumption}, the following stability results 
 for $g_1$ and $g_2^2$ hold,
 \begin{equation*}
 \begin{aligned}
 |g_1|_{H^{-1}(0,T)}&\le C \|\E[ I^{1-\alpha}_t h(\cdot,\o)]\|_{L^2(0,T)},\\ 
  |g_2^2|_{H^{-1}(0,T)}&\le C \|\V[ I^{1-\alpha}_t h(\cdot,\o)]\|_{L^2(0,T)}. 
  \end{aligned}
 \end{equation*}
 Here the constant $C>0$ is independent of $h$. 
\end{theorem}

With Theorem \ref{stability}, the following proposition about uniqueness 
can be derived immediately.
\begin{proposition}[Uniqueness]\label{uniqueness}
 Under Assumption \ref{assumption}, $g_1$ and $g_2^2$ can be uniquely determined by the moments
 $$\E[ I^{1-\alpha}_t h(t,\o)],
 \ \V[ I^{1-\alpha}_t h(t,\o)],\quad t\in[0,T].$$

 More precisely, suppose $g_1,g_2,\tilde{g}_1, \tilde{g}_2$ satisfy 
 Assumption \ref{assumption}, and denote
 the corresponding realizations as $h,\tilde{h}$, respectively. Then
 \begin{equation}\label{equality_1}
 \begin{aligned}
  \E[ I^{1-\alpha}_t h(t,\o)]&=\E[ I^{1-\alpha}_t \tilde{h}(t,\o)],\\
 \ \V[ I^{1-\alpha}_t h(t,\o)]&=\V[ I^{1-\alpha}_t \tilde{h}(t,\o)],
 \quad t\in[0,T],
 \end{aligned}
 \end{equation}
 leads to $g_1=\tilde{g}_1$, $g_2^2=\tilde{g}_2^2$ on $[0,T]$.
\end{proposition}

\subsection{Outline}
This paper is structured as follows. Section \ref{section:preliminary} 
includes some preliminary knowledge, such as the 
probability space $(\Omega,\mathcal{F},\mathbb{P})$, the 
Ito isometry formula and the maximum principles in fractional diffusion equations. 
Also, we define the stochastic weak 
solution $u$ as a stochastic process, $u:(0,T]\times\Omega \to L^2(D)$. 
In section \ref{section:main}, the proofs for Theorem \ref{uniqueness} and  
Proposition \ref{stability} are built. 
After the theoretical analysis, we consider the numerical reconstruction 
in section \ref{sec:numerical}. A mollification method is established  
and some numerical results are displayed.

\section{Preliminary setting}\label{section:preliminary}

\subsection{Eigensystem of $\A$ and the space $\mathcal{D}( \A^{\gamma} )$}
Since $\A$ is a positive elliptic operator defined on $H^2(D)\cap H^1_0(D)$, 
then its eigensystem $\{\lambda_n,\phi_n(x):n\in\mathbb{N}^+\}$ satisfies 
the following properties, 
\begin{itemize}
 \item $0<\lambda_1\le \lambda_2\le\cdots\le \lambda_n\le \cdots$ and 
 $\lambda_n\to \infty$ as $n\to \infty$,
 \item  $\{\phi_n(x):n\in\mathbb{N}^+\}\subset H^2(D)\cap H^1_0(D)$ 
 forms an orthonormal basis in $L^2(D)$.
\end{itemize}

Then given $\gamma>0,$ we define the space $\mathcal{D}(\A^\gamma)$ as 
\begin{equation*}
 \mathcal{D}(\A^\gamma):=\Big\{\psi\in L^2(D): \sum_{n=1}^\infty 
 \lambda_n^{2\gamma} \l\psi(\cdot),\phi_n(\cdot)\rd^2<\infty\Big\},
\end{equation*}
where $\l\cdot,\cdot\rd$ denotes the inner product in $L^2(D)$. 
It holds that $\mathcal{D}(\A^\gamma)\subset H^{2\gamma}(D)$ 
since $D$ has smooth boundary, \cite{SakamotoYamamoto:2011}. 

Moreover, we give the definition of Sobolev space with negative integer 
order $H^{-m}(0,T)$, which is used in Theorem \ref{stability}. 
$H^{-m}(0,T)$ is defined as the dual space of $H^m(0,T)$, namely 
$H^{-m}(0,T)=(H^m(0,T))'$, and from \cite[Theorem 12.1]{LionsMagenes:1972V1}, we have 
\begin{equation*}
 H^{-m}(0,T)=\Big\{\psi: \psi=\sum_{|j|\le m} D^j \psi_j,\ \exists \text{ some } \psi_j\in L^2(0,T)\Big\}.
\end{equation*}
Accordingly, define the $H^{-m}$ semi-norm as 
$$ | \psi |_{ H^{-m}(0,T) } = \| \psi_m \|_{ L^2(0,T) } .$$

\subsection{Probability setting and Ito isometry formula}
We give the definitions of probability space and the statistical 
moments $\E,\V$ below. 
\begin{definition}
 We call $(\Omega,\mathcal{F},\mathbb{P})$ a probability space 
 if $\Omega$ denotes the nonempty sample space, 
 $\mathcal{F}$ is the $\sigma-$algebra of $\Omega$ and 
 $\mathbb{P}:\mathcal{F}\to [0,1]$ is the probability measure. 
 
 For a random variable $X$ defined on $(\Omega,\mathcal{F},\mathbb{P})$, 
 the expectation $\E$ and variance $\V$ are defined as follows, 
 \begin{equation*}
  \E[X]=\int_\Omega X(\o)\ d\mathbb{P}(\o),\quad 
  \V[X]=\E[(X-E[X])^2].
 \end{equation*}
\end{definition}

For the time dependent Brownian motion, which is described as 
Wiener process $\W(t)$, we list some of its properties. 
\begin{remark}
 The Wiener process $\W(t)$ used in this article possesses 
 the following properties, 
 \begin{itemize}
  \item $\W(0)=0,$
  \item $\W(t)$ is almost surely continuous, 
  \item $\W(t)$ has independent increments and satisfies 
  $$\W(t)-\W(s)\sim \mathcal{N}(0,t-s),\quad 0\le s\le t,$$
  where $\mathcal{N}$ is the normal distribution. 
 \end{itemize}
\end{remark}

Now we can state the Ito isometry formula, which plays a crucial role  
in our analysis. In the following lemma, the random measure 
derived from $\W$ is denoted by $d\W(t)$.
\begin{lemma}{(\cite{Bernt2003stochastic}).}\label{Ito isometry formula} 
Let $(\Omega,\mathcal{F},\mathbb{P})$ be a probability space and let 
$\psi: [0,\infty)\times \Omega\rightarrow\mathbb{R}$
satisfy the following properties, 
\begin{itemize}
 \item[(1)] $(t,\omega)\rightarrow \psi(t,\omega)$ is 
 $\mathcal{B}\times\mathcal{F}$-measurable, 
 where $\mathcal{B}$ denotes the Borel $\sigma$-algebra on 
 $[0,\infty),$
 \item[(2)] $\psi(t,\omega)$ is $\mathcal{F}_t$-adapted,
 \item[(3)] $\E [\int_0^S\psi^2(\tau,\omega)\mathrm{d}\tau] <\infty$ 
 for $S>0$.
\end{itemize}
Then the Ito integral $\int_0^S \psi(\tau,\o)\ d\W(\tau)$ is well defined 
and it follows that
\begin{equation*}
 \E\Big[\Big(\int_0^S\psi(\tau,\omega)~d\W(\tau)\Big)^2\Big]
=\E\Big[\int_0^S\psi^2(\tau,\omega)~d\tau\Big].
\end{equation*}
\end{lemma}

\subsection{Stochastic weak solution}
Since we can not make sure $\dot{\W}(t)$ is continuous or differentiable 
for a given $\o\in\Omega$, we need to consider the weak solution 
in the sense of Ito integral. 

Firstly, we define the fractional integral operator as follows.
\begin{definition}
 The fractional integral operator $\I$ is given as
\begin{equation*}
 \I \psi(t)=\Gamma(\alpha)^{-1}\int_0^t(t-\tau)^{\alpha-1}\psi(\tau)\ d\tau,
 \quad \psi\in L^1_{loc}(0,\infty).
\end{equation*}
\end{definition}

The following remark explains that why we need to set the restriction 
$\alpha\in(1/2,1)$ on the fractional order $\alpha$. 
\begin{remark}
With the above definition, we can define the Ito integral 
$\I g(t,\omega)$ as  
 \begin{equation*}
  \I g(t,\omega)=\I g_1(t)+\Gamma(\alpha)^{-1}\int_0^t(t-\tau)^{\alpha-1}g_2(\tau)\ d\W(\tau).
 \end{equation*}
From the conditions $\alpha\in(1/2,1)$ and $g_2\in C[0,T]$, 
we have $(t-\tau)^{\alpha-1}g_2(\tau)$ is square-integrable 
on $[0,T].$ Then the well-definedness of the Ito integral 
$\int_0^t(t-\tau)^{\alpha-1}g_2(\tau)\ d\W(\tau)$ is ensured by 
Lemma \ref{Ito isometry formula}. 
\end{remark}

In addition, the direct calculation gives that 
\begin{equation*}
 \I \D \psi (t)=\psi(t)-\psi(0),
\end{equation*}
which can help us build the definition of the weak solution for 
equation \eqref{SDE}. 
\begin{definition}[Stochastic weak solution]\label{weak solution}
We say the stochastic process 
$u(\cdot,t,\omega):(0,T]\times\Omega\to L^2(D)$ is 
a stochastic weak solution of equation \eqref{SDE} if for each 
$\psi\in H^2(D)\cap H_0^1(D)$ and $\omega\in \Omega$, it holds that  
 \begin{equation*}
  \l u(\cdot,t,\o), \psi(\cdot)\rd+\l \I\A u(\cdot,t,\o), \psi(\cdot)\rd
  = \I g(t,\o)\ \l f(\cdot), \psi(\cdot)\rd,\quad t\in (0,T].
 \end{equation*}
\end{definition}

\subsection{Auxiliary lemmas}

The following lemmas will be used in the proof of stability.

\begin{lemma}{(\cite[Theorem 1.6]{Podlubny1999fractional})}\label{lem:mittag_bound}
\label{lem-ml-asymp}
The Mittag Leffler function $E_{\alpha,\beta}(z)$ is defined as 
$$E_{\alpha,\beta}(z)=\sum_{k=0}^\infty \frac{z^k}{\Gamma(\alpha k+\beta)}, 
\quad z\in\mathbb{C}.$$ 
Let $0<\alpha<2$ and $\beta\in\mathbb R$. 
Then for $\mu$ satisfying $\pi\alpha/2<\mu<\min\{\pi,\pi\alpha\}$, 
there exists a constant $C=C(\alpha,\beta,\mu)>0$ such that
$$
|E_{\alpha,\beta}(z)| \le \frac{C}{1+|z|},\quad \mu\le |\arg(z)| \le \pi.
$$
%Moreover, the Mittag-Leffler function $E_{\alpha,1}(-\lambda_n t^\alpha)$ is completely monotonic.
\end{lemma}

\begin{lemma}{(Maximum principle, \cite[Theorem 2]{Luchko:2009}).}\label{lem:max}
Fix $T\in (0,\infty),$ let $\psi$ satisfy the following 
fractional diffusion equation 
 \begin{equation*}\label{FDE}
    \D \psi+\A \psi=F(x,t),\quad (x,t)\in D\times(0,T],
 \end{equation*}
 and define $\Lambda_T=\partial D\times [0,T]\cup \overline{D}\times \{0\}$.
If $F\le 0$, then 
\begin{equation*}
 \psi(x,t)\le \max\{0,\max\{\psi(x,t):(x,t)\in \Lambda_T\}\},
 \quad (x,t)\in D\times (0,T].
\end{equation*}
\end{lemma}

The next lemma contains a $L^2$ regularity. 
\begin{lemma}\label{lem:vt}
 Let $v(x,t)$ be the solution of the following fractional diffusion equation, 
 \begin{equation}\label{FDE_v}
  \begin{cases}
   \begin{aligned}
    \D v+\A v &=0,&& (x,t)\in D\times(0,T],\\
    v(x,t)&=0,&& (x,t)\in \partial D\times(0,T],\\
    v(x,0)&=f(x),&& x\in D,
   \end{aligned}
  \end{cases}
 \end{equation}
 then there exists $C>0$ depending on $f(x),\alpha,T$ such that  
 $\|v_t(x_0,\cdot)\|_{L^2(0,T)}\le C$.  
\end{lemma}
\begin{proof}
 From \cite{SakamotoYamamoto:2011}, we can give the representation of 
 $v$ as 
 \begin{equation*}
  v(x,t)=\sum_{n=1}^\infty f_n E_{\alpha,1}(-\lambda_n t^\alpha) \phi_n(x), 
  \quad f_n=\l f(\cdot),\phi_n(\cdot)\rd,
 \end{equation*} 
and the regularity result $v\in C([0,T];H^2(D))$. 
Since $D\in \mathbb{R}^d,\ 1\le d\le 3$, from the Sobolev inequalities, 
we have that $v(\cdot,t)\in C(D)$ and 
$\|v(\cdot,t)\|_{C(D)}\le C\|v(\cdot,t)\|_{H^2(D)}$. 
By this continuous regularity, we can write $v(x_0,t)$ as the 
convergent series,  
\begin{equation*}
 v(x_0,t)=\sum_{n=1}^\infty  E_{\alpha,1}(-\lambda_n t^\alpha) 
 f_n\phi_n(x_0)
\end{equation*}
and denote its partial sum as $v_N(x_0,t)$, i.e. 
$v_N(x_0,t)=\sum_{n=1}^N E_{\alpha,1}(-\lambda_n t^\alpha) 
 f_n\phi_n(x_0)$. Recall the formula  
\begin{equation*}
 [E_{\alpha,1}(-\lambda_n t^\alpha)]'=-\lambda_n t^{\alpha-1} 
E_{\alpha,\alpha}(-\lambda_n t^\alpha),
\end{equation*}
and define 
\begin{equation*}
\begin{aligned}
 V(t)&=-\sum_{n=1}^\infty \lambda_n t^{\alpha-1} 
E_{\alpha,\alpha}(-\lambda_n t^\alpha) f_n\phi_n(x_0),\\
V_N(t)&=-\sum_{n=1}^N\lambda_n t^{\alpha-1} 
E_{\alpha,\alpha}(-\lambda_n t^\alpha) f_n\phi_n(x_0).
\end{aligned}
\end{equation*}
Next, we will show that given $\epsilon>0$, 
$v_t(x_0,t)=V(t)$ on $[\epsilon,T]$. 

For $t\in [\epsilon,T]$, the following estimate holds by Lemma 
\ref{lem:mittag_bound}, 
\begin{equation*}
 |\lambda_n t^{\alpha-1} E_{\alpha,\alpha}(-\lambda_n t^\alpha)| 
 \le C\frac{\lambda_n t^{\alpha-1}}{1+\lambda_n t^\alpha}\le C \epsilon^{-1}.
\end{equation*}
Then with Sobolev inequalities, we have 
\begin{equation*}
 \begin{aligned}
  |V(t)-V_N(t)|^2\le& \Big\|\sum_{n=N+1}^\infty \lambda_n t^{\alpha-1} 
E_{\alpha,\alpha}(-\lambda_n t^\alpha) f_n\phi_n(\cdot)\Big\|^2_{C(D)}\\
\le &C \Big\|\sum_{n=N+1}^\infty \lambda_n t^{\alpha-1} 
E_{\alpha,\alpha}(-\lambda_n t^\alpha) f_n\phi_n(\cdot)\Big\|^2_{H^2(D)}\\
\le& C \sum_{n=N+1}^\infty [\lambda_n t^{\alpha-1} 
E_{\alpha,\alpha}(-\lambda_n t^\alpha)]^2 \lambda_n^2f_n^2\\
\le &C\epsilon^{-2} \sum_{n=N+1}^\infty \lambda_n^2f_n^2.
 \end{aligned}
\end{equation*}
Since $\sum_{n=1}^\infty\lambda_n^2f_n^2 =\|f\|^2_{\mathcal{D}(\mathcal{A})}<\infty,$ 
the upper bound $C\epsilon^{-2} \sum_{n=N+1}^\infty \lambda_n^2f_n^2$ 
will converge to zero as $N\to \infty$, and note that it is independent 
of $t$. Then we can conclude that the series $V(t)$ is uniformly 
convergent on $[\epsilon,T]$. Realize that $V(t)$ is derived 
from $v(x_0,t)$ by termwise differentiation, then the uniform 
convergence of $V(t)$ and the convergence of $v(x_0,t)$ give that 
$V(t)=v_t(x_0,t)$ on $[\epsilon,T],\ \epsilon>0$.

Now we can show $v_t(x_0,t)\in L^2(0,T)$. For $t\in [\epsilon,T]$, 
\begin{equation*}
 \begin{aligned}
  |v_t(x_0,t)|^2&\le \Big\|\sum_{n=1}^\infty \lambda_n t^{\alpha-1} 
E_{\alpha,\alpha}(-\lambda_n t^\alpha) f_n\phi_n(\cdot)\Big\|^2_{C(D)}\\
  &\le C \Big\|\sum_{n=1}^\infty \lambda_n t^{\alpha-1} 
E_{\alpha,\alpha}(-\lambda_n t^\alpha) f_n\phi_n(\cdot)\Big\|^2_{H^2(D)}\\
&\le C t^{2\alpha-2}\sum_{n=1}^\infty 
[E_{\alpha,\alpha}(-\lambda_n t^\alpha)]^2 \lambda_n^4f_n^2\\
&\le C t^{2\alpha-2} \sum_{n=1}^\infty \lambda_n^4f_n^2
\le C \|f\|^2_{\mathcal{D}(\mathcal{A}^2)} t^{2\alpha-2}.
 \end{aligned}
\end{equation*}
Consequently, recalling that $\alpha\in(1/2,1)$, 
\begin{equation*}
 \begin{aligned}
  \|v_t(x_0,\cdot)\|^2_{L^2(\epsilon,T)}
  &\le C(f)\int_\epsilon^T t^{2\alpha-2}\ dt\\
  &=C(f,\alpha)(T^{2\alpha-1}-\epsilon^{2\alpha-1})\le C(f,\alpha,T)<\infty.
 \end{aligned}
\end{equation*}
Since $C(f,\alpha,T)$ is independent of the choice of $\epsilon$, we have 
\begin{equation*}
  \|v_t(x_0,\cdot)\|^2_{L^2(0,T)} 
  =\lim_{\epsilon\to0^+}  \|v_t(x_0,\cdot)\|^2_{L^2(\epsilon,T)}
  \le C(f,\alpha,T)<\infty, 
\end{equation*}
which completes the proof.
\end{proof}

\section{Main results}\label{section:main}
In this section we will give the proofs of Theorem \ref{stability} 
and Proposition \ref{uniqueness}. 

\subsection{The second kind Volterra equation}

The next lemma yields the representation of the weak solution $u(x,t,\o)$.
\begin{lemma}\label{lem:uv}
 Under Definition \ref{weak solution}, the weak solution $u$ 
 can be written as
 \begin{equation}\label{duhamel}
  \begin{aligned}
   u(x,t,\o)=\I g(t,\o) f(x)+\int_0^t \I g(\tau,\o) v_t(x,t-\tau)\ d\tau,
  \end{aligned}
 \end{equation}
 $v(x,t)$ satisfies equation \eqref{FDE_v}. 
 \end{lemma}

\begin{proof}
Note that
\begin{equation*}
 \A v_t=(\A v)_t=-\partial(\D v)/\partial t,
\end{equation*}
then we have
  \begin{equation*}
  \begin{aligned}
   &\I \A \int_0^t \I g(\tau,\o) v_t(x,t-\tau)\ d\tau\\
   =&-\I \int_0^t \I g(\tau,\o) \frac{\partial (\D v)}{\partial t}
   (x,t-\tau)\ d\tau\\
   =&-\Gamma(\alpha)^{-1}\int_0^t (t-s)^{\alpha-1}\int_0^s
   \I g(\tau,\o) \frac{\partial (\D v)}{\partial t}
   (x,s-\tau)\ d\tau\ ds\\
   =&-\Gamma(\alpha)^{-1}\int_0^t \I g(\tau,\o)\int_\tau^t
    (t-s)^{\alpha-1}\frac{\partial (\D v)}{\partial t}
   (x,s-\tau)\ ds\ d\tau.
   \end{aligned}
  \end{equation*}
From direct calculation, we have 
\begin{equation*}
\begin{aligned}
 \Gamma(\alpha)^{-1}\int_\tau^t (t-s)^{\alpha-1}
 \frac{\partial (\D v)}{\partial t}(x,s-\tau)\ ds
 &=\partial_t^{1-\alpha}(\D v)(x,t-\tau)\\
 &=v_t(x,t-\tau)-\Gamma(\alpha)^{-1}(t-\tau)^{\alpha-1}\D v(x,0)\\
 &=v_t(x,t-\tau)+\Gamma(\alpha)^{-1}(t-\tau)^{\alpha-1}\A f(x).
 \end{aligned}
\end{equation*}
Hence,
\begin{equation*}
 \begin{aligned}
  &\I \A \int_0^t \I g(\tau,\o) v_t(x,t-\tau)\ d\tau\\
  =&-\int_0^t \I g(\tau,\o)v_t(x,t-\tau)\ d\tau
   -\Gamma(\alpha)^{-1}\A f(x)
  \int_0^t \I g(\tau,\o) (t-\tau)^{\alpha-1}\ d\tau\\
  =&-\int_0^t \I g(\tau,\o)v_t(x,t-\tau)\ d\tau
   -(\I)^2 g(t,\o)\ \A f(x),
 \end{aligned}
\end{equation*}
which leads to
\begin{equation*}
\begin{aligned}
 \I\A u&=
 \I \A \int_0^t \I g(\tau,\o) v_t(x,t-\tau)\ d\tau
 +(\I)^2 g(t,\o)\ \A f(x)\\
 &=-\int_0^t \I g(\tau,\o)v_t(x,t-\tau)\ d\tau\\
 &=f\I g-u.
\end{aligned}
\end{equation*}
Moreover, we can derive $v_t(\cdot,t)\in L^2(D)$ from 
the analysis in \cite{SakamotoYamamoto:2011} and the condition 
$f\in \mathcal{D}(\mathcal{A}^2)$. Inserting these regularity estimates 
in \eqref{duhamel} yields that $u(\cdot,t,\o)\in L^2(D)$. 

Now we have $u$ in \eqref{duhamel} satisfies 
Definition \ref{weak solution} and complete the proof.
\end{proof}

With Lemmas \ref{Ito isometry formula} and \ref{lem:uv},
 the lemma below follows, which includes the unknowns $g_1,g_2$ 
and the statistical moments of $h(t,\o)$.

\begin{lemma}
Define 
 $$G_1(t)=\int_0^t g_1(\tau)\ d\tau,\quad  G_2(t)=\int_0^t g^2_2(\tau)\ d\tau,$$ 
 then it holds that for $t\in (0,T],$
\begin{equation}\label{integral equation}
 \begin{aligned}
  G_1(t)&=f^{-1}(x_0)\E[ I^{1-\alpha}_t h(t,\o)]-f^{-1}(x_0)\int_0^t G_1(\tau)v_t(x_0,t-\tau)\ d\tau,\\
  G_2(t)&=f^{-2}(x_0)\V[ I^{1-\alpha}_t h(t,\o)]-2f^{-2}(x_0)\int_0^t G_2(\tau) v(x_0,t-\tau)
   v_t(x_0,t-\tau)\ d\tau.
 \end{aligned}
\end{equation}
\end{lemma}
\begin{proof}
From \eqref{duhamel}, we can obtain the following result
\begin{equation*}
\begin{aligned}
 I^{1-\alpha}_t u(x,t,\o)
 =& \frac{f(x)}{\Gamma(\alpha)\Gamma(1-\alpha)}\int_0^t 
 (t-\tau)^{-\alpha}\int_0^\tau (\tau-s)^{\alpha-1} g(s,\o)\ ds\ d\tau\\
 &+ \frac{1}{\Gamma(1-\alpha)}\int_0^t (t-\tau)^{-\alpha}\int_0^\tau 
 \I g(\tau-s,\o) v_t(x,s)\ ds\ d\tau \\
 =& \frac{f(x)}{\Gamma(\alpha)\Gamma(1-\alpha)}\int_0^t g(s,\o)
 \int_s^t (t-\tau)^{-\alpha}(\tau-s)^{\alpha-1} \ d\tau\ ds\\
 &+ \frac{1}{\Gamma(1-\alpha)}\int_0^t v_t(x,s)\int_s^t 
 (t-\tau)^{-\alpha}\ \I g(\tau-s,\o) \ d\tau\ ds \\
=& \frac{1}{\Gamma(\alpha)\Gamma(1-\alpha)}\Bigg[ f(x)\int_0^t g(s,\o)
 \int_s^t (t-\tau)^{-\alpha}(\tau-s)^{\alpha-1} \ d\tau\ ds\\
 &+\int_0^t v_t(x,s)\int_0^{t-s} g(r,\o) \int_r^{t-s}(t-s-\tau)^{-\alpha}
 (\tau-r)^{\alpha-1} \ d\tau\ dr\ ds\Bigg].
 \end{aligned}
\end{equation*}
Due to 
\begin{equation*}
 \int_s^t (t-\tau)^{-\alpha}(\tau-s)^{\alpha-1} \ d\tau
 =B(1-\alpha,\alpha)=\Gamma(1-\alpha)\Gamma(\alpha)/\Gamma(1),
\end{equation*}
where $B$ is the Beta function, 
we have 
\begin{equation*}
\begin{aligned}
 I^{1-\alpha}_t u(x,t,\o)=&f(x)\int_0^t g(s,\o)\ ds
 +\int_0^t v_t(x,s)\int_0^{t-s} g(r,\o) \ dr\ ds\\
 =&f(x)\int_0^t g(s,\o)\ ds
 +\int_0^t g(r,\o) [v(x,t-r)-v(x,0)]\ dr\\
 =&\int_0^t g(\tau,\o) v(x,t-\tau)\ d\tau,
 \end{aligned}
\end{equation*}
i.e.
\begin{equation*}
 I^{1-\alpha}_t u(x,t,\o)=\int_0^t g_1(\tau) v(x,t-\tau)\ d\tau
 +\int_0^t g_2(\tau) v(x,t-\tau)\ d\W(\tau).
\end{equation*}
Hence,
\begin{equation}\label{equality_2}
 I^{1-\alpha}_t h(t,\o)=\int_0^t g_1(\tau) v(x_0,t-\tau)\ d\tau
 +\int_0^t g_2(\tau) v(x_0,t-\tau)\ d\W(\tau).
\end{equation} 
Applying Lemma \ref{Ito isometry formula} to \eqref{equality_2}, 
the following result can be derived,  
\begin{equation}\label{moment_unknown}
\begin{aligned}
 \E[ I^{1-\alpha}_t h(t,\o)]
 &=\int_0^t g_1(\tau)  v(x_0,t-\tau)\ d\tau,\\
 \V[ I^{1-\alpha}_t h(t,\o)]
 &=\int_0^t g_2^2(\tau) [ v(x_0,t-\tau)]^2\ d\tau,\quad t\in(0,T].
 \end{aligned}
\end{equation}
The right hand sides of above equations can be written as 
 \begin{equation*}
  \begin{aligned}
   \int_0^t g_1(\tau)  v(x_0,t-\tau)\ d\tau
   &=\int_0^t v(x_0,t-\tau)\ d(G_1(\tau))\\
   &=f(x_0)G_1(t)+\int_0^t G_1(\tau)v_t(x_0,t-\tau)\ d\tau,\\
   \int_0^t g_2^2(\tau) [ v(x_0,t-\tau)]^2\ d\tau
   &=\int_0^t [ v(x_0,t-\tau)]^2\ d(G_2(\tau))\\
   &=f^2(x_0)G_2(t)+2\int_0^t G_2(\tau) v(x_0,t-\tau)
   v_t(x_0,t-\tau)\ d\tau,
  \end{aligned}
 \end{equation*}
 which together with \eqref{moment_unknown} yields the desired result. 
\end{proof}

\subsection{Proof of Theorem \ref{stability} and Proposition \ref{uniqueness}}

\begin{lemma}\label{G1G2}
\begin{equation*}
\begin{aligned}
 \|G_1\|_{L^2(0,T)}&\le C \|\E[ I^{1-\alpha}_t h(\cdot,\o)]\|_{L^2(0,T)},\\ 
 \|G_2\|_{L^2(0,T)}&\le C \|\V[ I^{1-\alpha}_t h(\cdot,\o)]\|_{L^2(0,T)},
\end{aligned}
 \end{equation*}
 where the constant $C$ does not depend on $h$. 
\end{lemma}
\begin{proof} 
From \cite{Brunner:2017}, equation \eqref{integral equation} 
can be solved as 
\begin{equation*}
\begin{aligned}
 G_1(t)&=f^{-1}(x_0)\E[ I^{1-\alpha}_t h(t,\o)]+\int_0^t R_1(t-\tau) \E[ I^{1-\alpha}_t h(\tau,\o)]\ d\tau,\\
  G_2(t)&=f^{-2}(x_0)\V[ I^{1-\alpha}_t h(t,\o)]+\int_0^t R_2(t-\tau) \V[ I^{1-\alpha}_t h(\tau,\o)]\ d\tau,
  \end{aligned}
\end{equation*}
where $R_1, R_2$ are the resolvent kernels depending on 
$f(x_0),\ v(x_0,\cdot),\ v_t(x_0,\cdot)$.  
With Sobolev inequalities, the conditions $f\in H^4(D)$ and $D\subset
\mathbb{R}^d,\ d=1,2,3$ give that $\|f\|_{C(D)}\le C\|f\|_{H^4(D)}<\infty$. 
This and Lemma \ref{lem:max} ensure the boundedness of $|v(x_0,t)|$ 
on $[0,T]$. Also, from Lemma \ref{lem:vt}, we have 
$v_t(x_0,\cdot)\in L^2(0,T)$. Hence, by 
\cite[Theorem 8.3.3]{Brunner:2017}, the resolvent kernels 
$R_1,R_2$ both belong to the type $(L^2,T)$, namely,  
\begin{equation*}
 \int_0^T \int_0^t |R_j(t-s)|^2\ ds\ dt
 =\int_0^T \|R_j\|^2_{L^2(0,t)}\ dt <\infty,\quad j=1,2.
\end{equation*}

Now let's build the upper bounds of $G_1,G_2$. Holder inequality gives that  
\begin{equation*}
 \begin{aligned}
  \|G_1\|^2_{L^2(0,T)}\le& C(f) \|\E[ I^{1-\alpha}_t h(\cdot,\o)]\|^2_{L^2(0,T)} 
  +\int_0^T \|R_1\|^2_{L^2(0,t)}\ \|\E[ I^{1-\alpha}_t h(\cdot,\o)]\|^2_{L^2(0,t)}\ dt\\
  \le& \Big(C(f)+\int_0^T \|R_1\|^2_{L^2(0,t)}\ dt\Big)
  \ \|\E[ I^{1-\alpha}_t h(\cdot,\o)]\|^2_{L^2(0,T)}\\
  \le & C  \|\E[ I^{1-\alpha}_t h(\cdot,\o)]\|^2_{L^2(0,T)}.
 \end{aligned}
\end{equation*}
Analogously, we can prove that 
$ \|G_2\|_{L^2(0,T)}\le C \|\V[ I^{1-\alpha}_t h(\cdot,\o)]\|_{L^2(0,T)}$. 
The proof is complete.
\end{proof}

Now the stability and uniqueness can be proved. 
\begin{proof}[Proofs of Theorem \ref{stability} and 
Proposition \ref{uniqueness}]
 From Lemma \ref{G1G2} and the definition of $|\cdot|_{H^{-1}}$, we have 
 \begin{equation*}
 \begin{aligned}
 |g_1|_{H^{-1}(0,T)}&\le C \|\E[ I^{1-\alpha}_t h(\cdot,\o)]\|_{L^2(0,T)},\\ 
  |g_2^2|_{H^{-1}(0,T)}&\le C \|\V[ I^{1-\alpha}_t h(\cdot,\o)]\|_{L^2(0,T)}. 
  \end{aligned}
 \end{equation*}
 Also \eqref{moment_unknown} and Lemma \ref{G1G2} give that  
\begin{equation*}
\begin{aligned}
 |g_1-\tilde{g}_1|_{H^{-1}(0,T)}=|g_2^2-\tilde{g}_2^2|_{H^{-1}(0,T)}=0,  
\end{aligned}
\end{equation*}
which together with the continuities of $g_j, \tilde{g}_j, j=1,2$ 
leads to $g_1=\tilde{g}_1,\ g_2^2=\tilde{g}_2^2$.
\end{proof}

\section{Numerical reconstruction}\label{sec:numerical}
In this section, we illustrate the numerical reconstruction
of the unknowns $g_1, g_2^2$ from equation \eqref{integral equation}. 
Firstly we set  
$$ D \times [0,T] = [0,1]^2,\ \alpha = 0.8,\ x_0 = 1/2,
\ \A= -\Delta,\ f(x) = \sin{(\pi x)},$$  
then we consider the following experiments,  
\begin{align*}
 (e1):\ \ &g_1(t) = t + \sin{(2\pi t)} + \sin{(3\pi t)},
 \ g_2(t) = \left \{
 \begin{array}{lr}
     \sin{(2\pi t)} - 0.3, & t\in [0,1/2),\\
     \sin{(2\pi t)} + 0.3, & t\in [1/2, 1],
 \end{array}\right.  \\
 (e2):\ \ & g_1(t) = t + \sin{(2\pi t)} + \sin{(3\pi t)},
 \ g_2(t) = \sin{(\pi t)}, \\
 (e3):\ \ & g_1(t) = \left \{
 \begin{array}{lr}
     \sin{(2\pi t)} - 0.3, & t\in [0,1/2),\\
     \sin{(2\pi t)} + 0.3, & t\in [1/2, 1],
 \end{array} \right.
 \ g_2(t) = \sin{(\pi t)}, \\
(e4):\ \ &  g_1(t) = t + \sin{(2\pi t)} + \sin{(3\pi t)},
 \ g_2(t) =  \left \{
 \begin{array}{ll}
     4, & t\in [0,0.3),\\
     2, & t\in [0.3,0.6),\\
     1, & t\in [0.6, 1],
 \end{array} \right.\\
 (e5):\ \ & g_1(t) = \left \{
 \begin{array}{lr}
     \sin{(2\pi t)} - 0.3, & t\in [0,1/2),\\
     \sin{(2\pi t)} + 0.3, & t\in [1/2, 1],
 \end{array} \right.
\ g_2(t) = \left \{
 \begin{array}{ll}
     4, & t\in [0,0.3),\\
     2, & t\in [0.3,0.6),\\
     1, & t\in [0.6, 1].
 \end{array} \right.
\end{align*}
We add various size of white noises to our observation 
$h(t,\o)$, i.e. $ h^\sigma(t,\o) = h(t,\o) + \xi(t,\o),$ 
with Gaussian independent identically distributions 
$\xi(t,\o) \sim \sigma \mathcal N (0,1)$, i.e. $\xi(t,\o) \sim \mathcal N (0,\sigma^2)$.
Then we average a fractional integral of the point observations 
$h^\sigma(t,\o)$ to obtain the perturbed moments $\hat E^\sigma, \hat V^\sigma$, 
and display the numerical results under different $\sigma$ in 
Section \ref{sec:Num_ob}.
% And we adopt the parameters as
% $$ m =11, \quad N = 1000 ,$$
% the location of observation point
% $$ x_0 = \frac{\sqrt 2}{2} $$
% in all experiments.

\subsection{Direct solver}
To obtain the point observation $ h(t,\o) = u(x_0,t,\o) $, 
the fractional diffusion equation \eqref{SDE} needs to be solved, and 
to this end, the discretized scheme is constructed as follows. 

For spatial variation $x$, we apply a finite element approach using 
piecewise linear bases $\{\phi_j(x) \}_1^m$, namely 
$ \phi_j(x_k) = \delta_{jk},$ 
where $\{ x_j \}_1^m$ consists of a Delaunay triangulation of 
domain $D$. Then we define the finite element space 
as $\Vc_m=\sp\{\phi_j(x):j=1,\cdots,m\}$. The projections of $u$ and $f$ 
in $\Vc_m$ are defined as 
\begin{align*}
   \tilde{u}(x,t,\o) = \sum_{j=1}^m u(x_j,t,\o) \phi_j(x),
   \quad \tilde{f}(x) = \sum_{j=1}^m f(x_j) \phi_j(x).
\end{align*}
Then considering equation \eqref{SDE}, we have 
\begin{equation*}\label{eq:fem_SDE}
\sum_{j=1}^m \partial_t^\alpha u(x_j,t,\o) \l\phi_j, \phi_i\rd
+  \sum_{j=1}^m u(x_j,t,\o) \l\A \phi_j, \phi_i\rd = 
g(t,\o) \sum_{j=1}^m f(x_j) \l\phi_j, \phi_i\rd.
\end{equation*}
From the above equation, we define the mass matrix $\vec{M}$ 
and stiff matrix $\vec{S}$ w.r.t. basis $\{\phi_j\}_1^m $ as
\begin{align*}
    \vec{M} = \Big[\l\phi_j, \phi_i\rd \Big]_{i,j = 1}^m , 
    \quad   \vec{S} = \Big[ \l\A \phi_j, \phi_i\rd\Big]_{i,j = 1}^m ,
\end{align*}
which will be used to construct the discretized scheme for equation 
\eqref{SDE}.

For the discretization on time, the $L_1$-stepping scheme is used, 
which can be seen in 
\cite{JinLazarovZhou:2016,RundellZhang:2018}. 
Set the discrete time mesh $0 = t_0 < t_1 <\cdots < t_N = T$ and 
denote the time step size as $\Delta t = T/N$. 
Accordingly, the fractional derivate is approximated by
\begin{align*}
    & \partial_t^\alpha \psi(t_1) \approx b_{1,0} (\psi(t_1) - \psi(t_0)),\\
    & \partial_t^\alpha \psi(t_n) \approx \sum_{k=1}^{n-1} (b_{n,k-1} - b_{n,k}) \psi(t_k) 
    + b_{n,n-1}\psi(t_n) - b_{n,0} \psi(t_0) , 
    \quad n=2,\cdots, N,
\end{align*}
where parameters
$$ b_{n,k} = \Gamma(2-\alpha)^{-1} \Delta t^{-\alpha} [(n-k)^{1-\alpha} - (n-k-1)^{1-\alpha}] ,\quad k = 0,\cdots,n-1.$$

For the random term $g(t_n,\o) = g_1(t_n) + g_2(t_n) \dot \W(t_n)$, 
from the property $\W(t)-\W(s)\sim \mathcal{N}(t-s),$ 
the following approximation is given 
$$ \dot \W(t_n) \approx [\W(t_n)-\W(t_{n-1})]/\Delta t \sim 
\Delta t^{-1/2}\mathcal{N}(0,1).$$

Therefore, the discretized scheme for solving equation \eqref{SDE} 
is given as: for $\tilde{u}_n\in \Vc_m,\ n=1,\cdots N,$ its 
vector form $\vec{u}_n$ satisfies 
\begin{equation*}\label{eq:SDE_discrete}
\begin{aligned}
\left( b_{1,0}\vec{M} + \vec{S} \right) \vec{u}_1
=& \vec{M}  \Big( g_1(t_1)\vec f+ g_2(t_1) 
\Delta t^{-1/2}\mathcal{N}(0,1)\vec f+b_{1,0}\vec{u}_0\Big),\\
\left( b_{n,n-1}\vec{M} + \vec{S} \right) \vec{u}_n
=& \vec{M}  \Big( g_1(t_n)\vec f+ g_2(t_n)\Delta t^{-1/2} \mathcal{N}(0,1)\vec f
+\sum_{k=1}^{n-1}(b_{n,k}-b_{n,k-1})\vec{u}_k +b_{n,0}\vec{u}_0\Big),
\end{aligned}
\end{equation*}
noting that $\vec{u}_0=0$ due to the zero initial condition. 
To capture the randomness in system \eqref{SDE}, 
we need to collect numerous realizations for the single point 
solution $u(x_0,t,\o)$. Furthermore, the solution $v(x,t)$ 
defined in equation \eqref{FDE_v} can be numerically simulated 
analogously.

\subsection{Inverse problem and mollification}\label{sec:Num_mol}
Here we consider the inverse problem of solving the unknowns $g_1, g_2^2$ 
and denote the numerical approximations as $\hat g_1,  \hat{g}_2^2$ respectively. 

After measuring the noisy data $h ^\sigma(t,\o)$ disturbed by white noise 
$\sigma$ from $h(t,\o) = u(x_0, t,\o)$, we obtain the mean 
and variance value of the fractional integral $H^\sigma := 
I_t^{1-\alpha} h^\sigma(t,\o)$, denoting as $ E^\sigma(t)$ 
and $ V^\sigma(t)$, respectively. 
To reconstruct $G_1 $ and $G_2$ from \eqref{integral equation}, we need to solve
the second kind Volterra equation
\begin{equation}\label{eq:VoltII}
    Y(t) = X(t) + \int_0^t X(\tau) K(t- \tau) d\tau \triangleq X(t) + \K X(t),\ t \in [0,T].
\end{equation}
One of the common choices is the iterative method with initial guess 
$Y(t)$, namely 
\begin{equation}\label{iteration}
 X_{n+1}(t)=Y(t)-\K X_{n}(t),\ \ X_0(t)=Y(t).
\end{equation}
The convergence of \eqref{iteration} can be assured if 
$Y,K\in L^2[0,T]$, \cite{Tricomi-38}.

Before using iteration \eqref{iteration},  
we introduce the mollification method here. As defined in 
\cite{adams2003sobolev,MurioMejia-35}, a \textbf{mollifier} is a 
nonnegative, real-valued function belong to $C_0^\infty(\R)$. For example we take
$$ J(t) = \left\{
\begin{split}
    & c\cdot exp[ -1/(1-|t|^2) ], & |t|<1, \\
    & 0 ,&|t| \geq 1,
\end{split}\right.
$$
where $c>0$ is chosen so that $\int_{\R} J(t) dt = 1$. Then for any $\epsilon>0$,
$J_\epsilon(t) = \epsilon^{-1} J(t/ \epsilon)$ is a mollifier with compact support belonging to $(-\epsilon, \epsilon)$. The convolution
$$ J_\epsilon *\phi(t) = \int_\R J_\epsilon(t-\tau) \phi(\tau) d\tau $$
is called a mollification or regularization of $\phi$ if the right 
side converges. The convergence of $J_\epsilon * \phi$ as $\epsilon \to 0$ is assured by the regularity of $\phi$.
Back to the moments $E^\sigma(t), V^\sigma(t),\ t\in [0,T]$, 
in order to apply the mollification, we extend the domain 
by symmetric $2T$-periodic extension, denoted by 
$\hat E^\sigma(t), \hat V^\sigma(t)$, i.e.
$$ \hat E^\sigma(t) = \left\{
\begin{split}
    &  E^\sigma(t - 2k T),  & t\in [2k T, (2k+1)T], \\
    &  E^\sigma(2k T - t),  & t\in [(2k-1)T, 2k T],
\end{split}\right.  \quad  k\in\Z,
$$
$$ \hat V^\sigma(t) = \left\{
\begin{split}
    &  V^\sigma(t - 2k T),  & t\in [2k T, (2k+1)T], \\
    &  V^\sigma(2k T - t),  & t\in [(2k-1)T, 2k T],
\end{split}\right.   \quad  k\in\Z.
$$
Therefore, we can solve \eqref{eq:VoltII} and \eqref{iteration} to 
get $\int_0^t \hat g_1(\tau) d\tau$ and $\int_0^t \hat g_2^2(\tau) d\tau$ by letting 
\begin{equation*}
\begin{aligned}
 Y(t)&=f^{-1}(x_0)\ J_\epsilon*\hat E^\sigma(t),\ K(t) = f^{-1}(x_0) v_t(x_0,t),\ \text{or}\\
 Y(t)&=f^{-2}(x_0)\ J_\epsilon*\hat V^\sigma(t),\ K(t) = 2f^{-2}(x_0) v(x_0,t)v_t(x_0,t),
\end{aligned}
\end{equation*}
respectively.

Next we're going to state the meaning of using mollification.
Before that, to measure the accuracy of the reconstruction results, 
we give the $L^2$ error as, 
\begin{align*}
   er(g_1) &= \|g_1-\hat g_1\|_{ L^2[0,T] }
   \approx \Big(  \frac{\Delta t}{2} \sum_{i=0}^N p_i 
   \left| g_1(t_i)-\hat g_1(t_i) \right|^2  \Big)^\frac{1}{2}, \\
er(|g_2|) &= \big\| |g_2|- |\hat g_2| \big\|_{ L^2[0,T] }
   \approx \Big(  \frac{\Delta t}{2} \sum_{i=0}^N p_i 
   \left( |g_2(t_i)|-|\hat g_2(t_i)| \right)^2  \Big)^\frac{1}{2},
\end{align*}
where weights $p_0=p_N=1$, $p_i= 2,\ i=1,\cdots,N-1$. 
It is natural that the more realizations we use, the better the numerical 
results are. This is confirmed by 
Table \ref{tab:sample&err}\footnote{In this subsection, we set 
$\sigma =0$, i.e. no observation noise.}. However, large amount of samples causes 
high cost, as a consequence, the mollification is applied. In 
Table \ref{tab:moleps&err}, we can see the improvement caused by 
mollification. Hence, mollification can help us generate satisfactory 
results with relative less samples. Moreover, taking $(e1)$ as an example, 
Figures \ref{fig:E&g1&sample}, \ref{fig:E&g1&moleps}, \ref{fig:V&g2&sample} and \ref{fig:V&g2&moleps} 
provide visual evidence to support the necessity of mollification. 
\begin{table}[h!]
  \centering
  \caption{Errors without mollification.}
  \label{tab:sample&err}
  \begin{tabular}{c|c|c|c|c|c|c}
    \hline
    \multirow{2}{*}{} & \multicolumn{2}{c|}{$10^3$ samples} & \multicolumn{2}{c|}{$10^4$ samples} & \multicolumn{2}{c}{ $10^5$ samples} \\
    \hline
     & $er(g_1)$ & $er(|g_2|) $  & $er(g_1)$ & $er(|g_2|) $  & $er(g_1)$ & $er(|g_2|) $ \\
    \hline
    $(e1)$ & 0.299794 & 0.056932  &  0.091171 & 0.018218 & 0.028705 & 0.009218 \\
    $(e2)$ & 0.449422 & 0.084018  &  0.142783 & 0.029317 & 0.046238 & 0.012576 \\
    $(e3)$ & 0.447672 & 0.083309  &  0.231421 & 0.043802 & 0.081429 & 0.020119 \\
    $(e4)$ & 1.570322 & 0.288016  &  0.503009 & 0.104327 & 0.151976 & 0.055286 \\
    $(e5)$ & 1.542568 & 0.310195  &  0.509957 & 0.115601 & 0.171542 & 0.057249 \\
    \hline
  \end{tabular}
\end{table}
\begin{table}[h!]
  \centering
  \caption{Errors for different $\epsilon$, $10^3$ realizations used.}
  \label{tab:moleps&err}
  \begin{tabular}{c|c|c|c|c|c|c}
    \hline
    \multirow{2}{*}{} & \multicolumn{2}{c|}{ Without mollification} & \multicolumn{2}{c|}{ $\epsilon=0.005$ } & \multicolumn{2}{c}{ $\epsilon=0.05$  } \\
    \hline
     & $er(g_1)$ & $er(|g_2|) $  & $er(g_1)$ & $er(|g_2|) $  & $er(g_1)$ & $er(|g_2|) $ \\
    \hline
    $(e1)$ & 0.299794 & 0.056932  & 0.124498 & 0.030254 & 0.050135 & 0.055231 \\
    $(e2)$ & 0.449422 & 0.084018  & 0.203604 & 0.039741 & 0.097088 & 0.034309 \\
    $(e3)$ & 0.447672 & 0.083309  & 0.216069 & 0.041439 & 0.096336 & 0.033850 \\
    $(e4)$ & 1.570322 & 0.288016  & 0.827782 & 0.172930 & 0.266984 & 0.262352 \\
    $(e5)$ & 1.542568 & 0.310195  & 0.624018 & 0.195870 & 0.165958 & 0.281415 \\
    \hline
  \end{tabular}
\end{table}

\begin{figure}[htbp]
  \centering

  \subfigure[ $\hat E$ with $10^3$ realizations.]{
  \begin{minipage}[t]{0.33\textwidth}
  \centering
  \includegraphics[width=\textwidth]{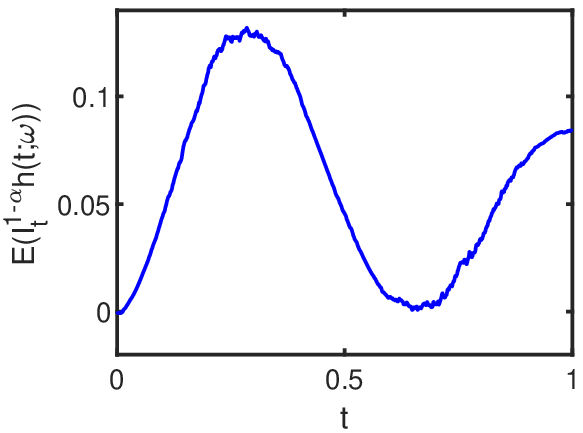}
  %\caption{fig1}
  \end{minipage}%
  }%
  \subfigure[ $\hat E$ with $10^4$ realizations.]{
  \begin{minipage}[t]{0.33\textwidth}
  \centering
  \includegraphics[width=\textwidth]{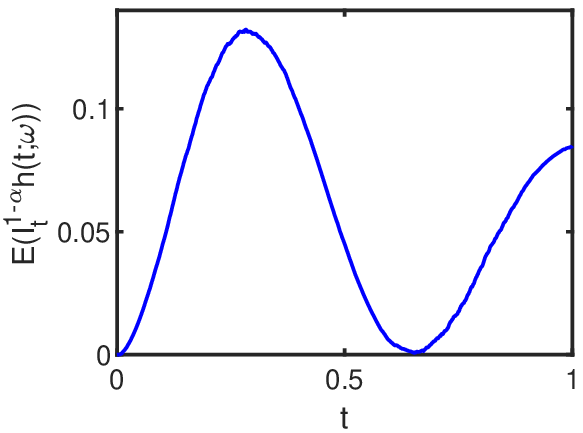}
  %\caption{fig2}
  \end{minipage}%
  }%
  \subfigure[ $\hat E$ with $10^5$ realizations.]{
  \begin{minipage}[t]{0.33\textwidth}
  \centering
  \includegraphics[width=\textwidth]{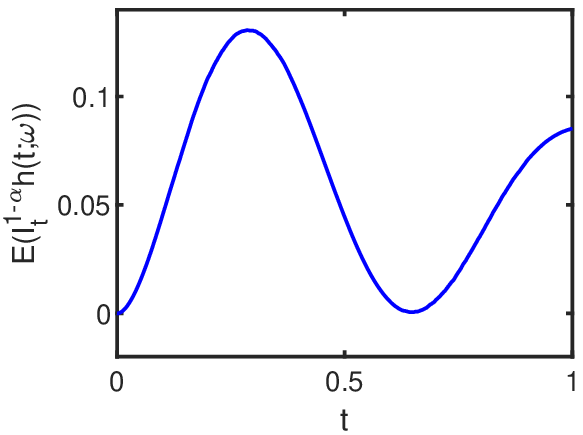}
  %\caption{fig2}
  \end{minipage}%
  }%
  \quad
  \subfigure[ $\hat g_1$ with $10^3$ realizations.]{
  \begin{minipage}[t]{0.33\textwidth}
  \centering
  \includegraphics[width=\textwidth]{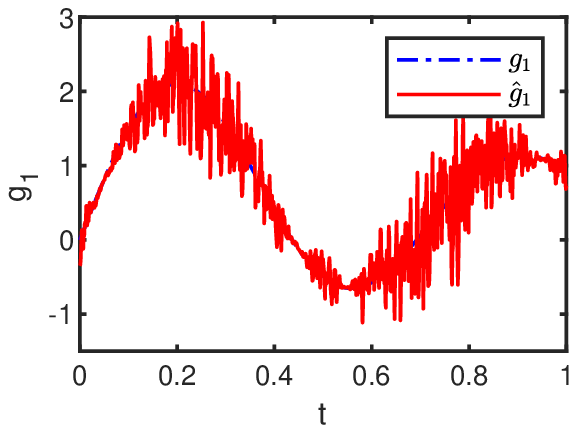}
  %\caption{fig1}
  \end{minipage}%
  }%
  \subfigure[ $\hat g_1$ with $10^4$ realizations.]{
  \begin{minipage}[t]{0.33\textwidth}
  \centering
  \includegraphics[width=\textwidth]{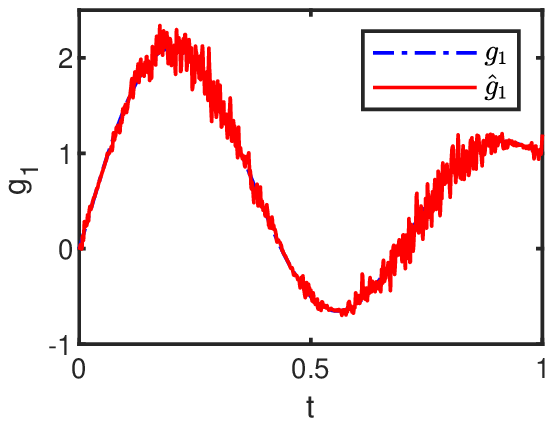}
  %\caption{fig2}
  \end{minipage}%
  }%
  \subfigure[ $\hat g_1$ with $10^5$ realizations.]{
  \begin{minipage}[t]{0.33\textwidth}
  \centering
  \includegraphics[width=\textwidth]{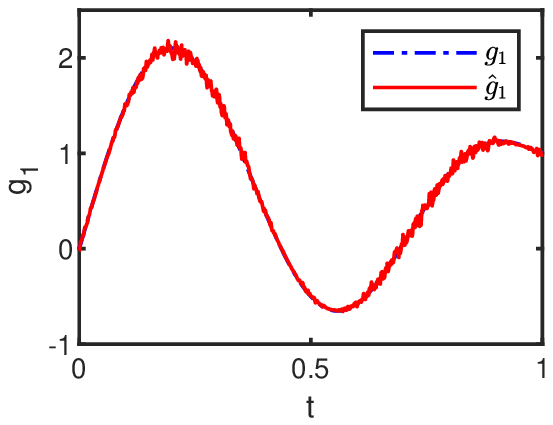}
  %\caption{fig2}
  \end{minipage}%
  }%

  \centering
  \caption{Experiment $(e1)$, $\hat E$ (top) and $g_1, \hat g_1$ 
  (bottom) for different amount of samples.}\label{fig:E&g1&sample}
\end{figure}
\begin{figure}[htbp]
  \centering

  \subfigure[ $\hat E$ without mollification.]{
  \begin{minipage}[t]{0.33\textwidth}
  \centering
  \includegraphics[width=\textwidth]{e5_E_1000.eps}
  %\caption{fig1}
  \end{minipage}%
  }%
  \subfigure[ $\hat E, J_\epsilon * \hat E$ when $\epsilon = 0.005$. ]{
  \begin{minipage}[t]{0.33\textwidth}
  \centering
  \includegraphics[width=\textwidth]{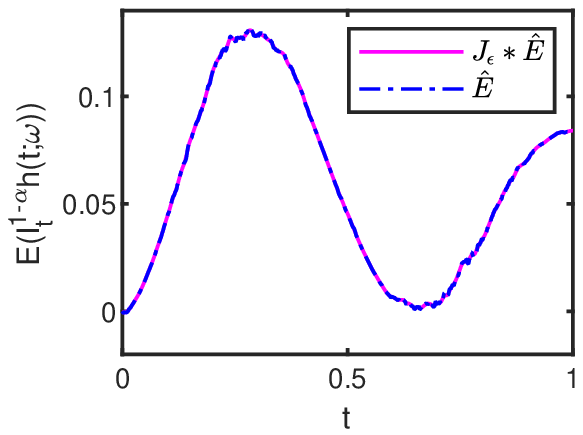}
  %\caption{fig2}
  \end{minipage}%
  }%
  \subfigure[ $\hat E, J_\epsilon * \hat E$ when $\epsilon = 0.05$. ]{
  \begin{minipage}[t]{0.33\textwidth}
  \centering
  \includegraphics[width=\textwidth]{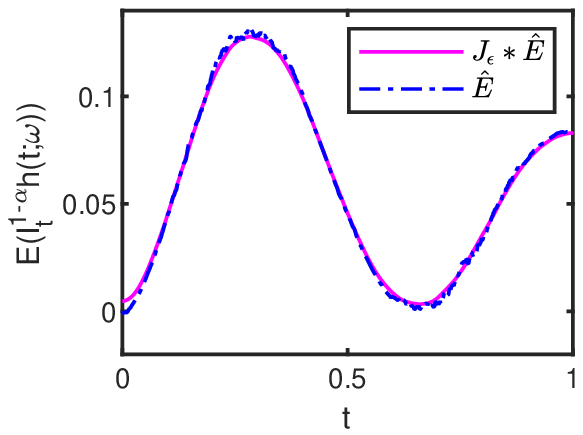}
  %\caption{fig2}
  \end{minipage}%
  }%
  \quad
  \subfigure[ $\hat g_1$ without mollification. ]{
  \begin{minipage}[t]{0.33\textwidth}
  \centering
  \includegraphics[width=\textwidth]{e5_g1_1000.eps}
  %\caption{fig1}
  \end{minipage}%
  }%
  \subfigure[ $\hat g_1$ when $\epsilon = 0.005$. ]{
  \begin{minipage}[t]{0.33\textwidth}
  \centering
  \includegraphics[width=\textwidth]{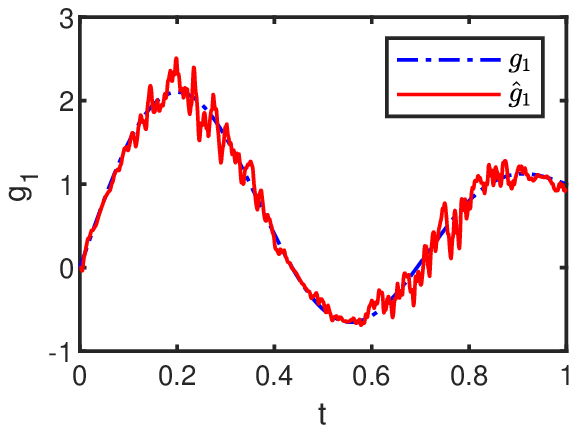}
  %\caption{fig2}
  \end{minipage}%
  }%
  \subfigure[ $\hat g_1$ when $\epsilon = 0.05$. ]{
  \begin{minipage}[t]{0.33\textwidth}
  \centering
  \includegraphics[width=\textwidth]{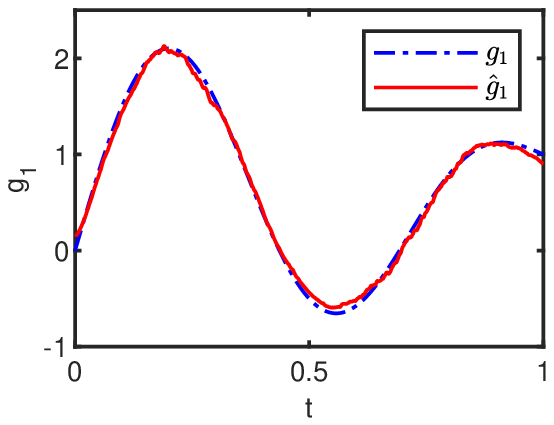}
  %\caption{fig2}
  \end{minipage}%
  }%

  \centering
  \caption{Experiment $(e1),$ $\hat E,J_\epsilon * \hat E$ (top) and 
  $g_1, \hat g_1$ (bottom) for different $\epsilon$, $10^3$ realizations.}
  \label{fig:E&g1&moleps}
\end{figure}
\begin{figure}[htbp]
  \centering

  \subfigure[ $\hat V$ with $10^3$ realizations.]{
  \begin{minipage}[t]{0.33\textwidth}
  \centering
  \includegraphics[width=\textwidth]{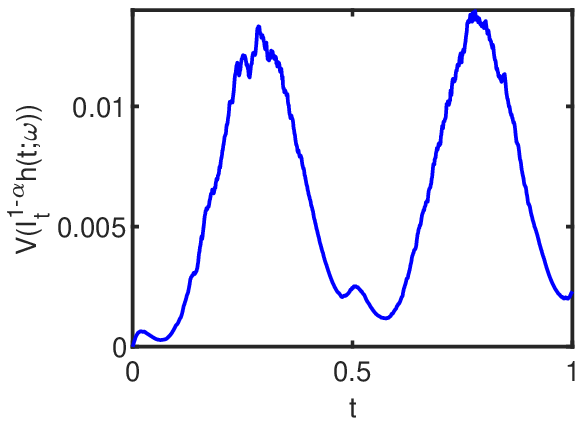}
  %\caption{fig1}
  \end{minipage}%
  }%
  \subfigure[ $\hat V$ with $10^4$ realizations.]{
  \begin{minipage}[t]{0.33\textwidth}
  \centering
  \includegraphics[width=\textwidth]{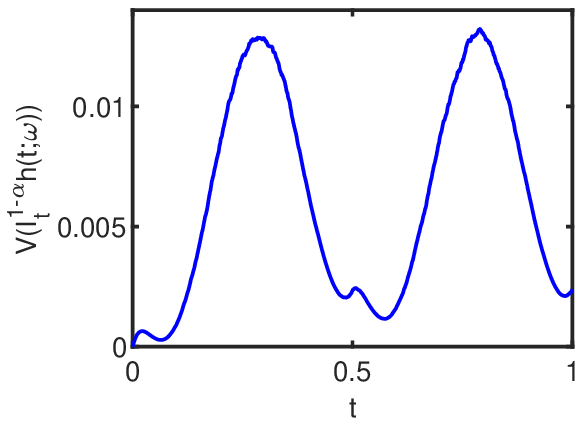}
  %\caption{fig2}
  \end{minipage}%
  }%
  \subfigure[ $\hat V$ with $10^5$ realizations.]{
  \begin{minipage}[t]{0.33\textwidth}
  \centering
  \includegraphics[width=\textwidth]{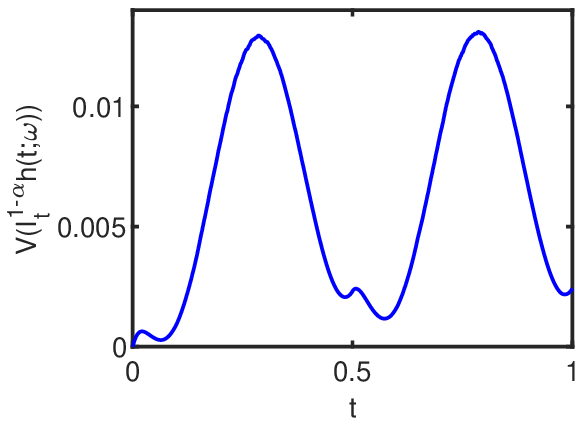}
  %\caption{fig2}
  \end{minipage}%
  }%
  \quad
  \subfigure[ $|\hat g_2|$ with $10^3$ realizations.]{
  \begin{minipage}[t]{0.33\textwidth}
  \centering
  \includegraphics[width=\textwidth]{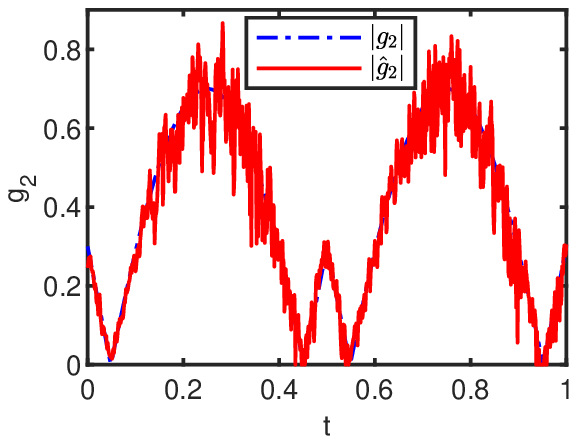}
  %\caption{fig1}
  \end{minipage}%
  }%
  \subfigure[ $|\hat g_2|$ with $10^4$ realizations.]{
  \begin{minipage}[t]{0.33\textwidth}
  \centering
  \includegraphics[width=\textwidth]{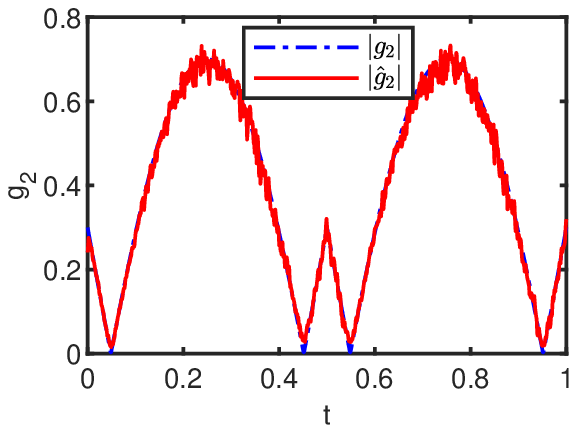}
  %\caption{fig2}
  \end{minipage}%
  }%
  \subfigure[ $|\hat g_2|$ with $10^5$ realizations.]{
  \begin{minipage}[t]{0.33\textwidth}
  \centering
  \includegraphics[width=\textwidth]{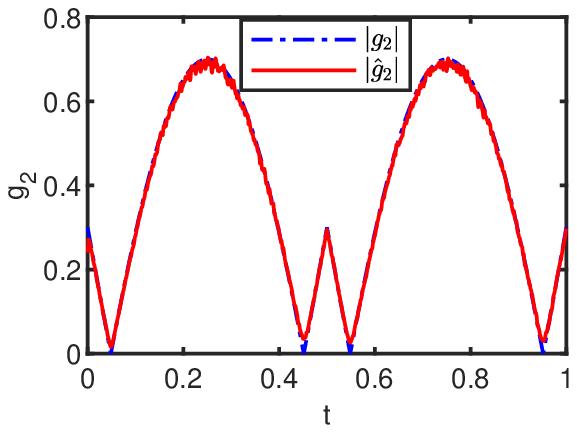}
  %\caption{fig2}
  \end{minipage}%
  }%

  \centering
  \caption{Experiment $(e1)$, $\hat V$ (top) and $|g_2|,|\hat g_2|$ (bottom) 
  for different amount of samples. }\label{fig:V&g2&sample}
\end{figure}
\begin{figure}[htbp]
  \centering

  \subfigure[ $\hat V$ without mollification.]{
  \begin{minipage}[t]{0.33\textwidth}
  \centering
  \includegraphics[width=\textwidth]{e5_V_1000.eps}
  %\caption{fig1}
  \end{minipage}%
  }%
  \subfigure[ $\hat V, J_\epsilon * \hat V$ when $\epsilon = 0.005$. ]{
  \begin{minipage}[t]{0.33\textwidth}
  \centering
  \includegraphics[width=\textwidth]{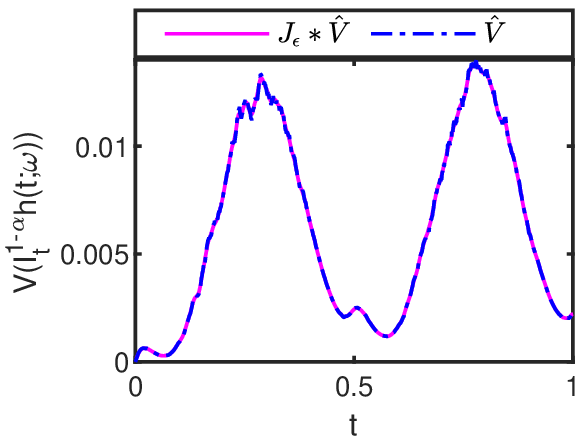}
  %\caption{fig2}
  \end{minipage}%
  }%
  \subfigure[ $\hat V, J_\epsilon * \hat V$ when $\epsilon = 0.05$. ]{
  \begin{minipage}[t]{0.33\textwidth}
  \centering
  \includegraphics[width=\textwidth]{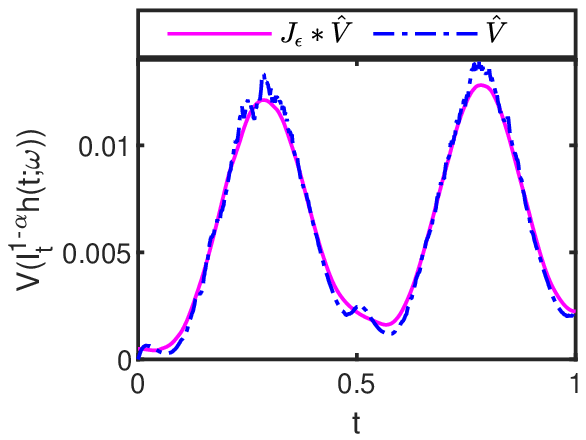}
  %\caption{fig2}
  \end{minipage}%
  }%
  \quad
  \subfigure[ $|\hat g_2|$ without mollification. ]{
  \begin{minipage}[t]{0.33\textwidth}
  \centering
  \includegraphics[width=\textwidth]{e5_g2_1000.eps}
  %\caption{fig1}
  \end{minipage}%
  }%
  \subfigure[ $|\hat g_2|$ when $\epsilon = 0.005$. ]{
  \begin{minipage}[t]{0.33\textwidth}
  \centering
  \includegraphics[width=\textwidth]{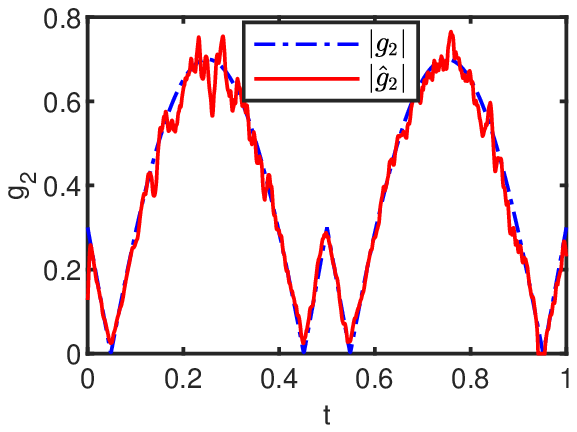}
  %\caption{fig2}
  \end{minipage}%
  }%
  \subfigure[ $|\hat g_2|$ when $\epsilon = 0.05$. ]{
  \begin{minipage}[t]{0.33\textwidth}
  \centering
  \includegraphics[width=\textwidth]{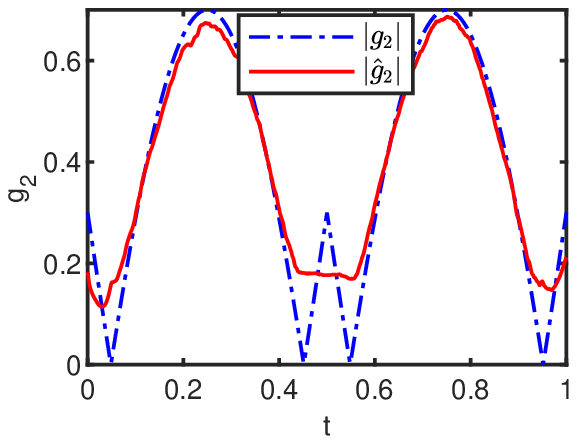}
  %\caption{fig2}
  \end{minipage}%
  }%

  \centering
  \caption{Experiment $(e1)$, $\hat V, J_\epsilon *\hat V$ (top) and 
  $|g_2|, |\hat g_2|$ (bottom) for different $\epsilon$, $10^3$ 
  realizations.}
  \label{fig:V&g2&moleps}
\end{figure}

Now, we are interested in the dependence of the reconstruction results 
on the mollification parameter $\epsilon$. Still using $(e1)$, as is shown in Figure 
\ref{fig:moleps&err}, as the mollification parameter increases, 
the error decreases at the beginning but then grows up. This is because 
large $\epsilon$ will lead to oversmoothing, making the moments 
deviating from the true values substantially. 
More precisely, we choose different values marked on Figure 
\ref{fig:moleps&err} to illustrate that too small or too large 
values of $\epsilon$ are not desirable, see Figure \ref{fig:g&moleps}.

\begin{figure}[htbp]
  \centering

  \subfigure[ $er(g_1)$ with various $\epsilon$.]{
  \begin{minipage}[t]{0.5\textwidth}
  \centering
  \includegraphics[width=\textwidth]{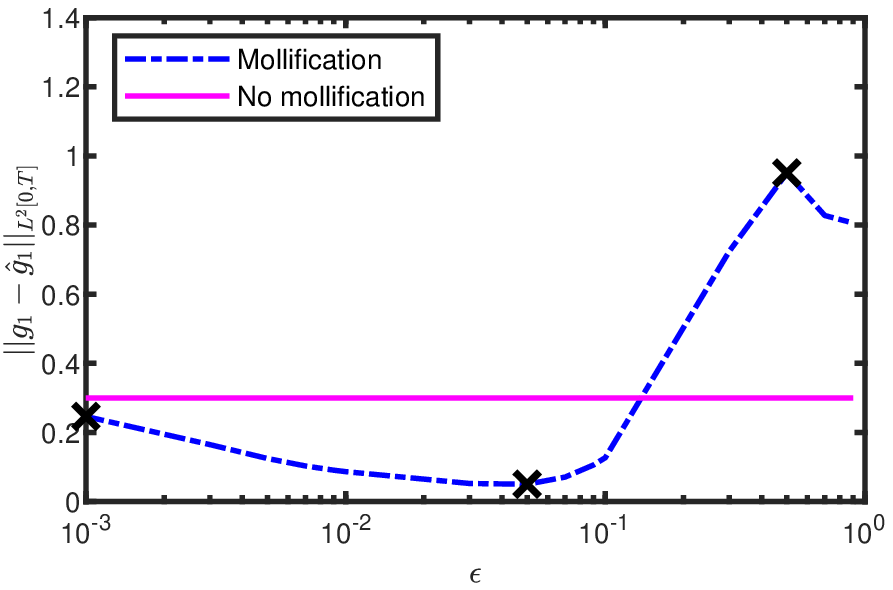}
  %\caption{fig1}
  \end{minipage}%
  }%
  \subfigure[ $er(|g_2|)$ with various $\epsilon$.]{
  \begin{minipage}[t]{0.5\textwidth}
  \centering
  \includegraphics[width=\textwidth]{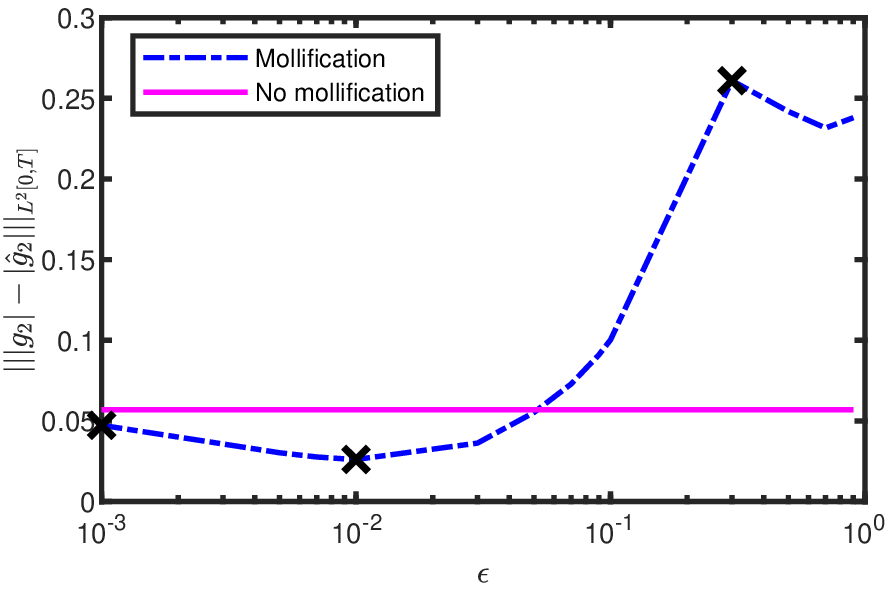}
  %\caption{fig2}
  \end{minipage}%
  }%

  \centering
  \caption{Experiment $(e1)$, $er(g_1)$ and $er(|g_2|)$ under different 
  $\epsilon$.}\label{fig:moleps&err}
\end{figure}
\begin{figure}[htbp]
  \centering

  \subfigure[ $\hat g_1$ when $\epsilon=0.001$. ]{
  \begin{minipage}[t]{0.33\textwidth}
  \centering
  \includegraphics[width=\textwidth]{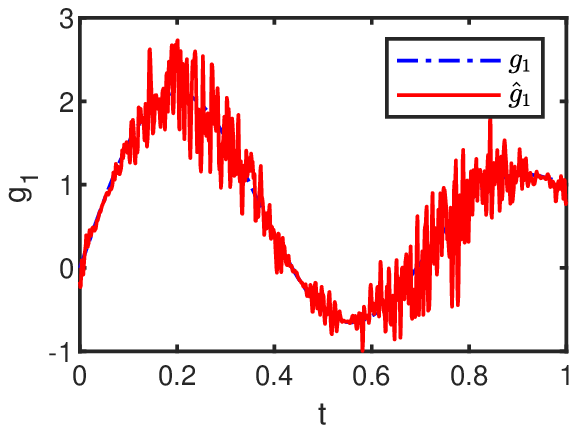}
  %\caption{fig1}
  \end{minipage}%
  }%
  \subfigure[$\hat g_1$ when $\epsilon=0.05$.]{
  \begin{minipage}[t]{0.33\textwidth}
  \centering
  \includegraphics[width=\textwidth]{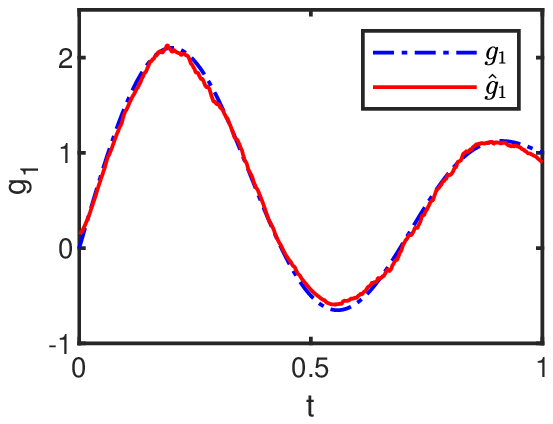}
  %\caption{fig2}
  \end{minipage}%
  }%
  \subfigure[$\hat g_1$ when $\epsilon=0.5$.]{
  \begin{minipage}[t]{0.33\textwidth}
  \centering
  \includegraphics[width=\textwidth]{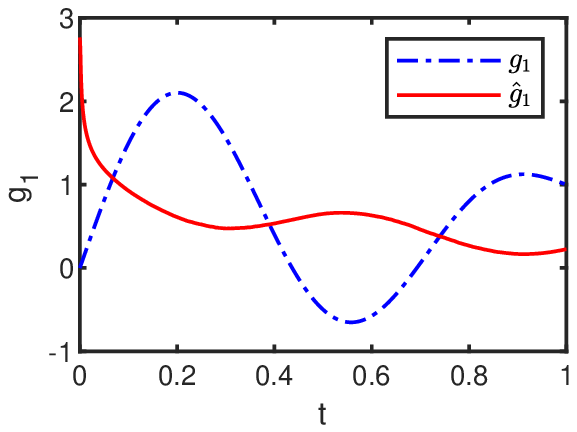}
  %\caption{fig2}
  \end{minipage}%
  }%
  \quad
  \subfigure[ $|\hat g_2|$ when $\epsilon=0.001$. ]{
  \begin{minipage}[t]{0.33\textwidth}
  \centering
  \includegraphics[width=\textwidth]{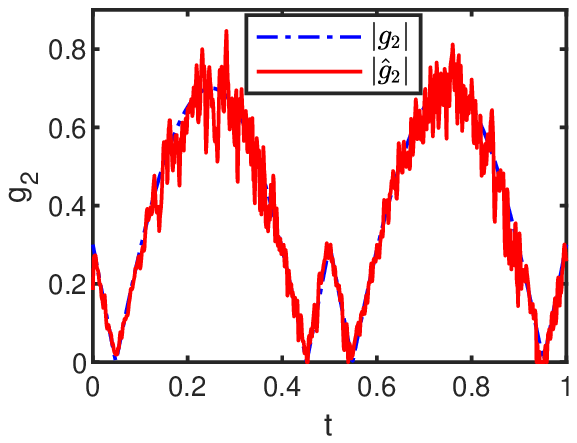}
  %\caption{fig1}
  \end{minipage}%
  }%
  \subfigure[ $|\hat g_2|$ when $\epsilon=0.01$. ]{
  \begin{minipage}[t]{0.33\textwidth}
  \centering
  \includegraphics[width=\textwidth]{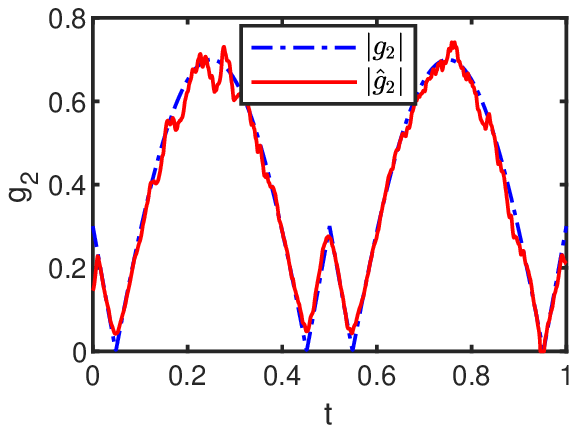}
  %\caption{fig2}
  \end{minipage}%
  }%
  \subfigure[ $|\hat g_2|$ when $\epsilon=0.3$. ]{
  \begin{minipage}[t]{0.33\textwidth}
  \centering
  \includegraphics[width=\textwidth]{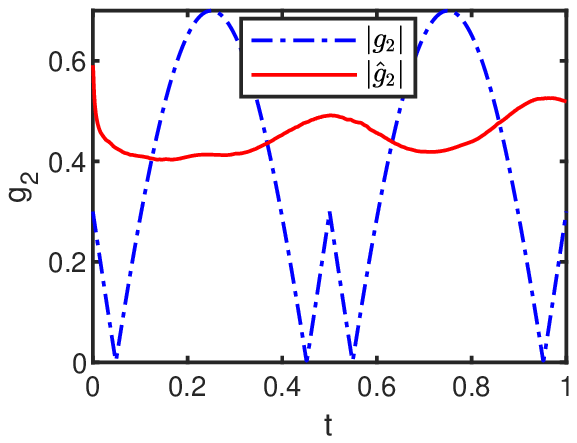}
  %\caption{fig2}
  \end{minipage}%
  }%

  \centering
  \caption{Experiment $(e1)$, $g_1, \hat g_1$ (top) and 
  $|g_2|, |\hat g_2|$ (bottom) for different $\epsilon$.}
  \label{fig:g&moleps}
\end{figure}

\subsection{Numerical experiments}\label{sec:Num_ob}
In this subsection, we will focus on the 
adaptation of our algorithm to different sizes of observation noise. 
Mostly the observation noises are raised by system observations, and 
obey a normal distribution, denoted as 
$\xi \sim \mathcal N(0,\sigma^2)$. As stated in \cite{DeFinetti-36}, 
from the Bayesian standpoint, the probability of the random variable 
$\xi \sim \mathcal N(0,\sigma^2)$ belongs to the interval 
$ [-2.58\sigma, 2.58\sigma]$ is 99\%. Thus, we will use the bound 
$2.58\sigma$ to evaluate the size of observation noises.

Figures \ref{fig:e1_ob&err}-\ref{fig:e5_g&ob} give 
the reconstructions for experiments $(ej),\ 1\le j\le5$. 
We can see that our algorithm performs well even for 
discontinuous cases. This means the continuous restrictions on the 
unknowns $g_1,g_2$ may be weakened in the numerical aspect. 
Note that the amount of samples is $10^3$ and we adopt 
$\epsilon = 0.05$, which is chosen according to the time-step 
$\Delta t=10^{-3}$.
\begin{figure}[htbp]
  \centering

  \subfigure[ $er(g_1)$ with various $\sigma$. ]{
  \begin{minipage}[t]{0.5\textwidth}
  \centering
  \includegraphics[width=0.9\textwidth]{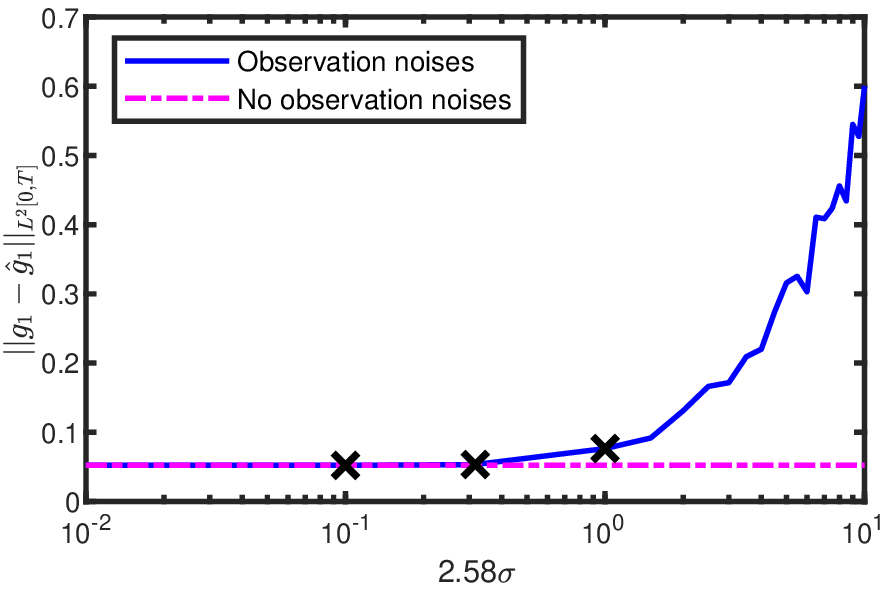}
  %\caption{fig1}
  \end{minipage}%
  }%
  \subfigure[ $er(|g_2|)$ with various $\sigma$. ]{
  \begin{minipage}[t]{0.5\textwidth}
  \centering
  \includegraphics[width=0.9\textwidth]{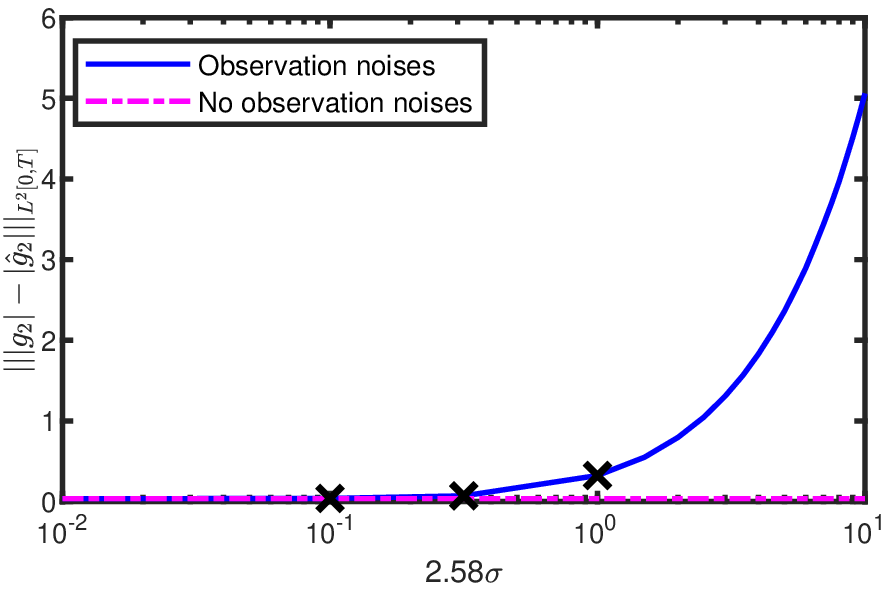}
  %\caption{fig2}
  \end{minipage}%
  }%

  \centering
  \caption{ Experiment $(e1)$, $er(g_1)$ (left) and $er(|g_2|)$ (right) 
  for different $\sigma$.}\label{fig:e1_ob&err}
\end{figure}
\begin{figure}[htbp]
  \centering

  \subfigure[$\hat g_1$ when $2.58\sigma =10^{-1}$. ]{
  \begin{minipage}[t]{0.33\textwidth}
  \centering
  \includegraphics[width=\textwidth]{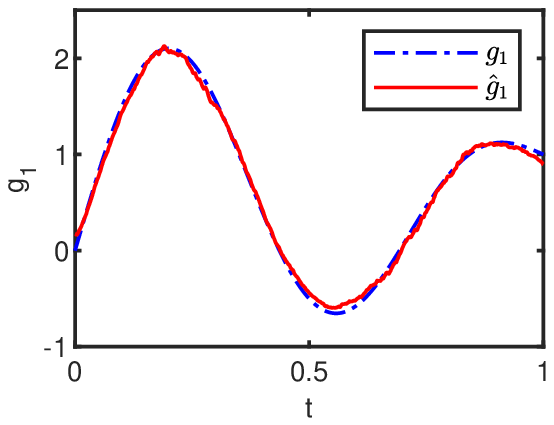}
  %\caption{fig1}
  \end{minipage}%
  }%
  \subfigure[ $\hat g_1$ when $2.58\sigma =10^{-\frac{1}{2}}$.  ]{
  \begin{minipage}[t]{0.33\textwidth}
  \centering
  \includegraphics[width=\textwidth]{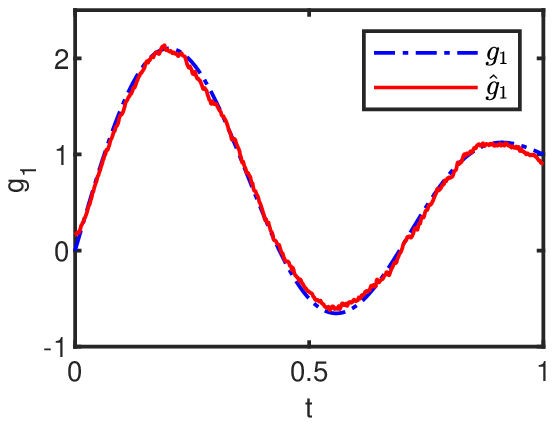}
  %\caption{fig2}
  \end{minipage}%
  }%
  \subfigure[ $\hat g_1$ when $2.58\sigma =1 $.  ]{
  \begin{minipage}[t]{0.33\textwidth}
  \centering
  \includegraphics[width=\textwidth]{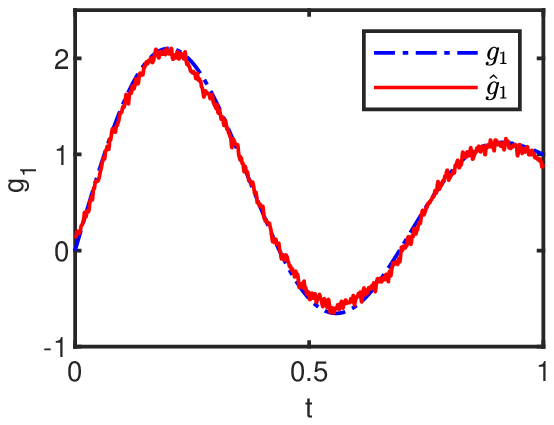}
  %\caption{fig2}
  \end{minipage}%
  }%
  \quad
  \subfigure[$|\hat g_2|$ when $2.58\sigma =10^{-1} $. ]{
  \begin{minipage}[t]{0.33\textwidth}
  \centering
  \includegraphics[width=\textwidth]{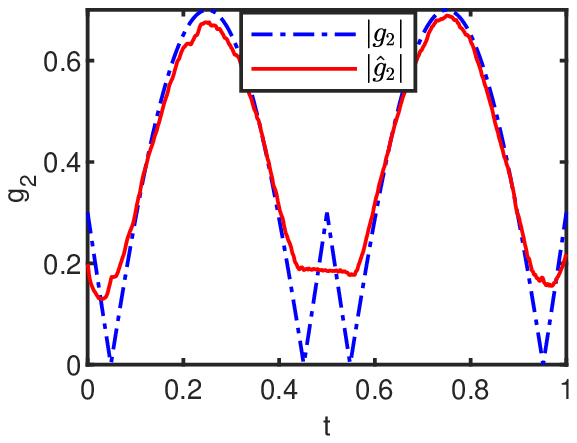}
  %\caption{fig1}
  \end{minipage}%
  }%
  \subfigure[$|\hat g_2|$ when $2.58\sigma = 10^{-\frac{1}{2}} $. ]{
  \begin{minipage}[t]{0.33\textwidth}
  \centering
  \includegraphics[width=\textwidth]{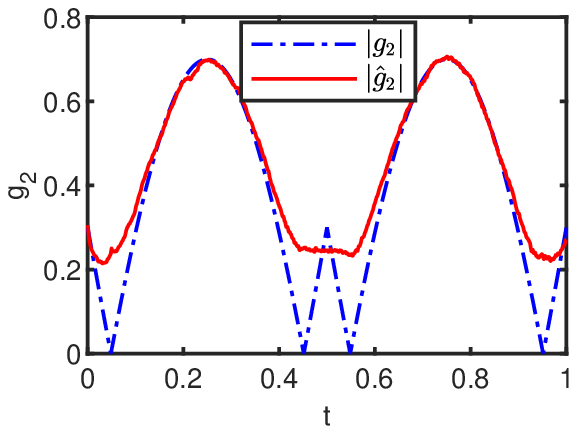}
  %\caption{fig2}
  \end{minipage}%
  }%
  \subfigure[$|\hat g_2|$ when $2.58\sigma =1 $. ]{
  \begin{minipage}[t]{0.33\textwidth}
  \centering
  \includegraphics[width=\textwidth]{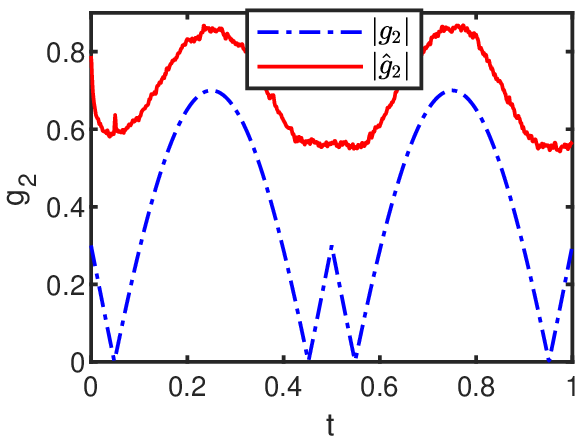}
  %\caption{fig2}
  \end{minipage}%
  }%

  \centering
  \caption{ Experiment $(e1)$, $g_1, \hat g_1$ (top) and 
  $|g_2|, |\hat g_2|$ (bottom) for different $\sigma$.}\label{fig:e1_g&ob}
\end{figure}
\begin{figure}[htbp]
  \centering

  \subfigure[ $er(g_1)$ with various $\sigma$. ]{
  \begin{minipage}[t]{0.5\textwidth}
  \centering
  \includegraphics[width=\textwidth]{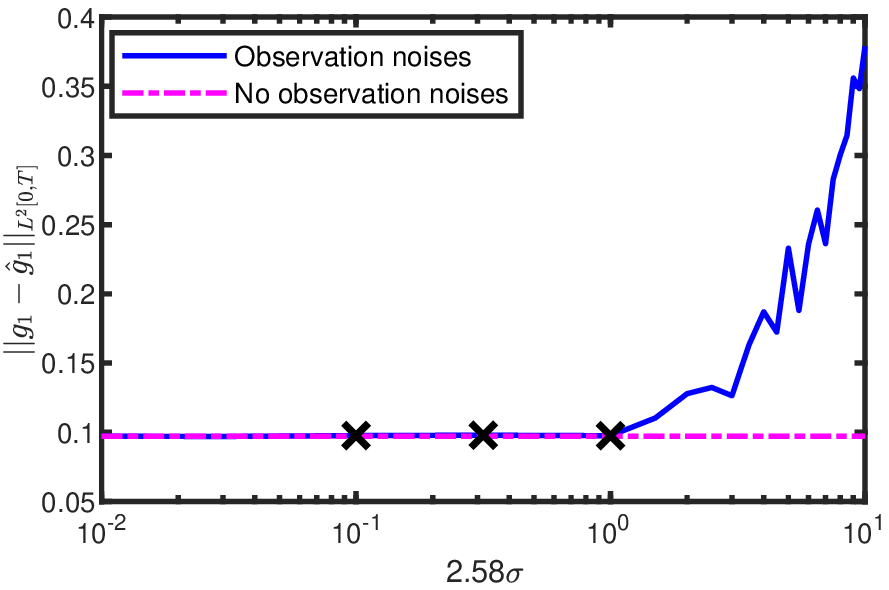}
  %\caption{fig1}
  \end{minipage}%
  }%
  \subfigure[ $er(|g_2|)$ with various $\sigma$. ]{
  \begin{minipage}[t]{0.5\textwidth}
  \centering
  \includegraphics[width=\textwidth]{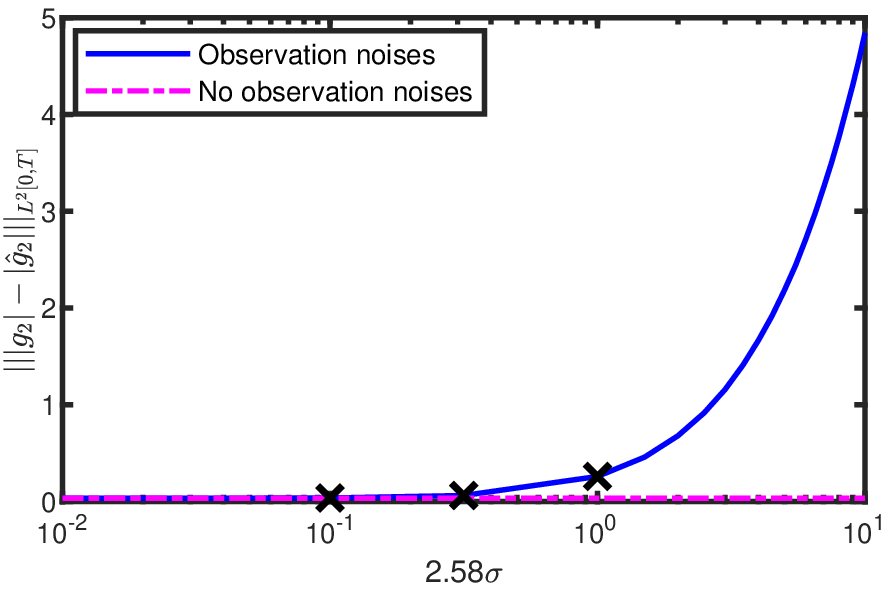}
  %\caption{fig2}
  \end{minipage}%
  }%

  \centering
  \caption{ Experiment $(e2)$, $er(g_1)$ (left) and $er(|g_2|)$ (right) 
  for different $\sigma$.}\label{fig:e2_ob&err}
\end{figure}
\begin{figure}[htbp]
  \centering

  \subfigure[$\hat g_1$ when $2.58\sigma =10^{-1}$. ]{
  \begin{minipage}[t]{0.33\textwidth}
  \centering
  \includegraphics[width=\textwidth]{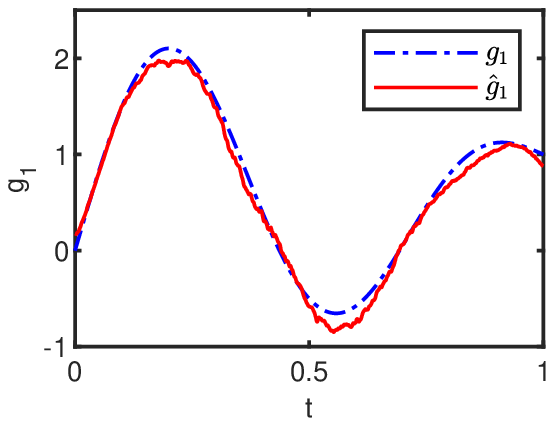}
  %\caption{fig1}
  \end{minipage}%
  }%
  \subfigure[ $\hat g_1$ when $2.58\sigma =10^{-\frac{1}{2}}$.  ]{
  \begin{minipage}[t]{0.33\textwidth}
  \centering
  \includegraphics[width=\textwidth]{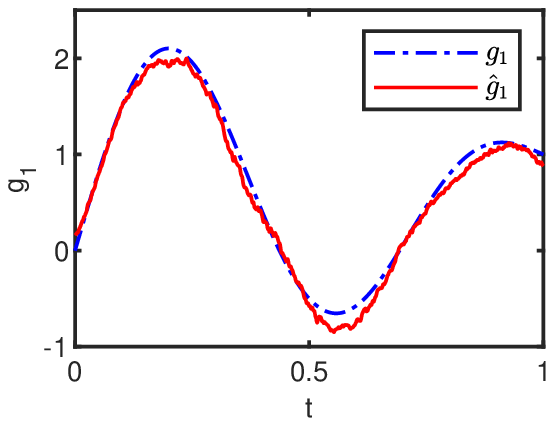}
  %\caption{fig2}
  \end{minipage}%
  }%
  \subfigure[ $\hat g_1$ when $2.58\sigma =1 $.  ]{
  \begin{minipage}[t]{0.33\textwidth}
  \centering
  \includegraphics[width=\textwidth]{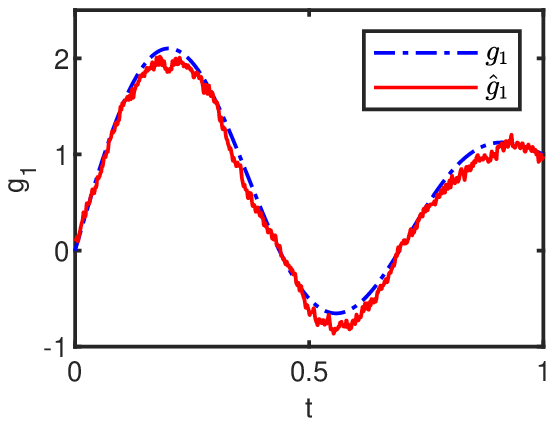}
  %\caption{fig2}
  \end{minipage}%
  }%
  \quad
  \subfigure[$|\hat g_2|$ when $2.58\sigma =10^{-1} $. ]{
  \begin{minipage}[t]{0.33\textwidth}
  \centering
  \includegraphics[width=\textwidth]{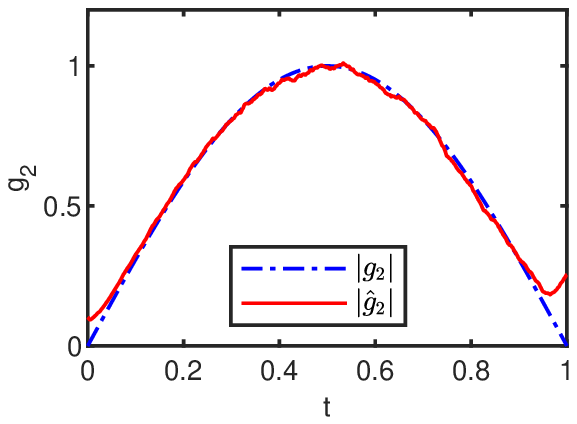}
  %\caption{fig1}
  \end{minipage}%
  }%
  \subfigure[$|\hat g_2|$ when $2.58\sigma = 10^{-\frac{1}{2}} $. ]{
  \begin{minipage}[t]{0.33\textwidth}
  \centering
  \includegraphics[width=\textwidth]{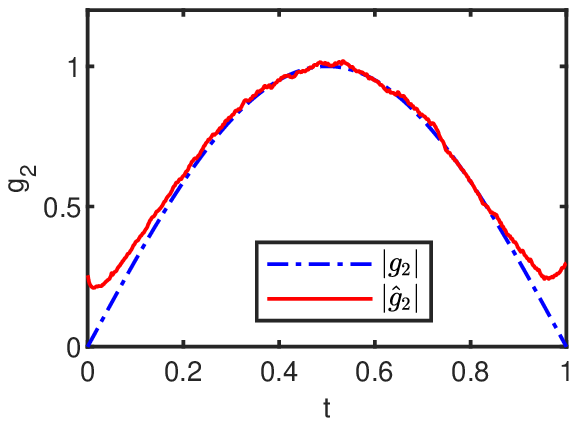}
  %\caption{fig2}
  \end{minipage}%
  }%
  \subfigure[$|\hat g_2|$ when $2.58\sigma =1 $. ]{
  \begin{minipage}[t]{0.33\textwidth}
  \centering
  \includegraphics[width=\textwidth]{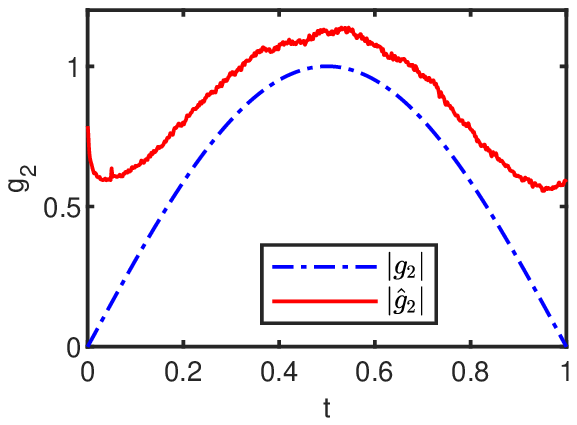}
  %\caption{fig2}
  \end{minipage}%
  }%

  \centering
  \caption{ Experiment $(e2)$, $g_1, \hat g_1$ (top) and 
  $|g_2|, |\hat g_2|$ (bottom) for different $\sigma$. }\label{fig:e2_g&ob}
\end{figure}
%
%
%{\red Although we require $g_1,|g_2|\in C[0,T]$, in fact, our algorithm can also apply to discontinuous functions, as example $(e3),(e4),(e5)$ show. }
%
%%
\begin{figure}[htbp]
  \centering

  \subfigure[ $er(g_1)$ with various $\sigma$. ]{
  \begin{minipage}[t]{0.5\textwidth}
  \centering
  \includegraphics[width=\textwidth]{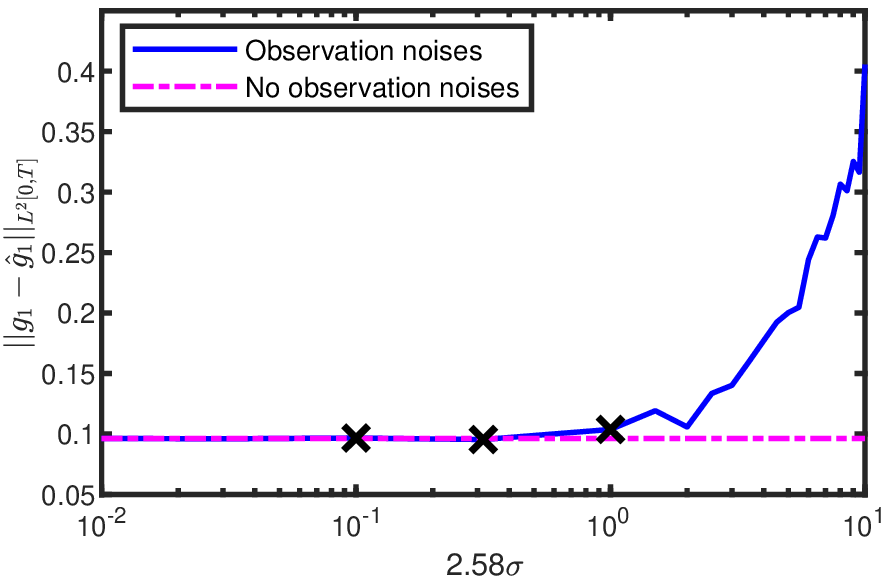}
  %\caption{fig1}
  \end{minipage}%
  }%
  \subfigure[ $er(|g_2|)$ with various $\sigma$. ]{
  \begin{minipage}[t]{0.5\textwidth}
  \centering
  \includegraphics[width=\textwidth]{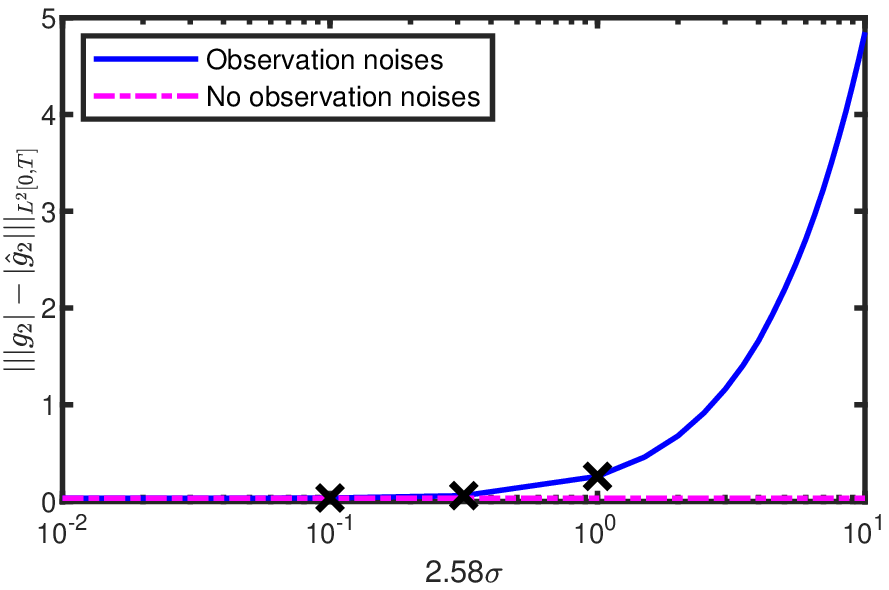}
  %\caption{fig2}
  \end{minipage}%
  }%

  \centering
  \caption{ Experiment $(e3)$, $er(g_1)$ (left) and $er(|g_2|)$ (right) 
  for different $\sigma$.}\label{fig:e3_ob&err}
\end{figure}
\begin{figure}[htbp]
  \centering

  \subfigure[$\hat g_1$ when $2.58\sigma =10^{-1}$. ]{
  \begin{minipage}[t]{0.33\textwidth}
  \centering
  \includegraphics[width=\textwidth]{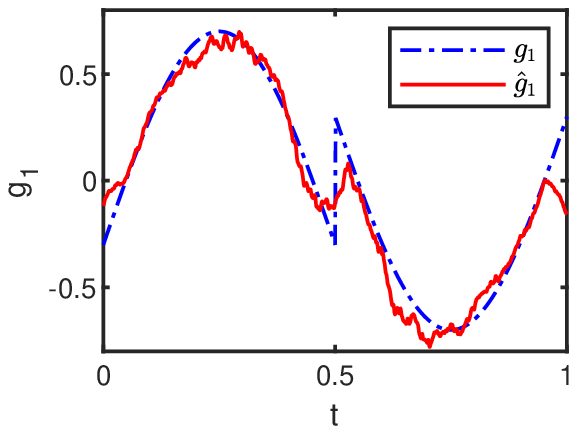}
  %\caption{fig1}
  \end{minipage}%
  }%
  \subfigure[ $\hat g_1$ when $2.58\sigma =10^{-\frac{1}{2}}$.  ]{
  \begin{minipage}[t]{0.33\textwidth}
  \centering
  \includegraphics[width=\textwidth]{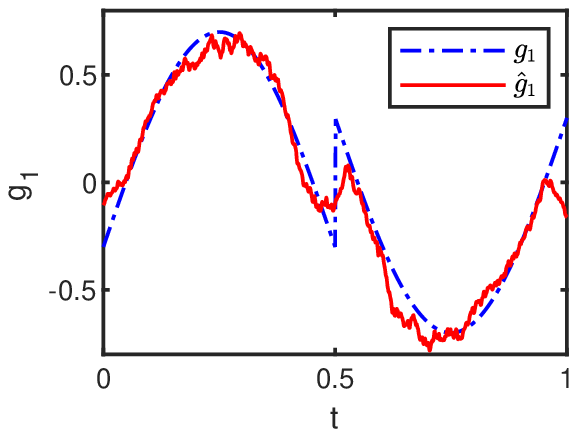}
  %\caption{fig2}
  \end{minipage}%
  }%
  \subfigure[ $\hat g_1$ when $2.58\sigma =1 $.  ]{
  \begin{minipage}[t]{0.33\textwidth}
  \centering
  \includegraphics[width=\textwidth]{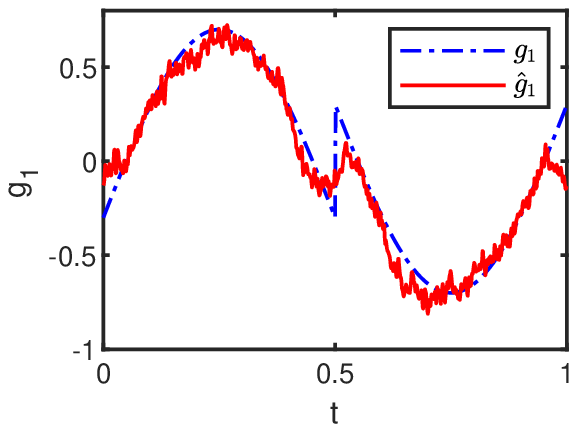}
  %\caption{fig2}
  \end{minipage}%
  }%
  \quad
  \subfigure[$|\hat g_2|$ when $2.58\sigma =10^{-1} $. ]{
  \begin{minipage}[t]{0.33\textwidth}
  \centering
  \includegraphics[width=\textwidth]{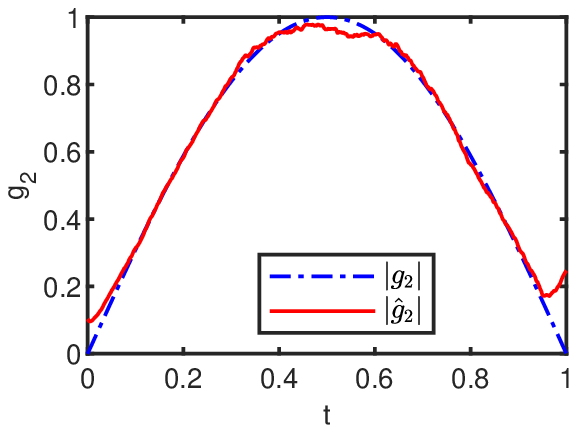}
  %\caption{fig1}
  \end{minipage}%
  }%
  \subfigure[$|\hat g_2|$ when $2.58\sigma = 10^{-\frac{1}{2}} $. ]{
  \begin{minipage}[t]{0.33\textwidth}
  \centering
  \includegraphics[width=\textwidth]{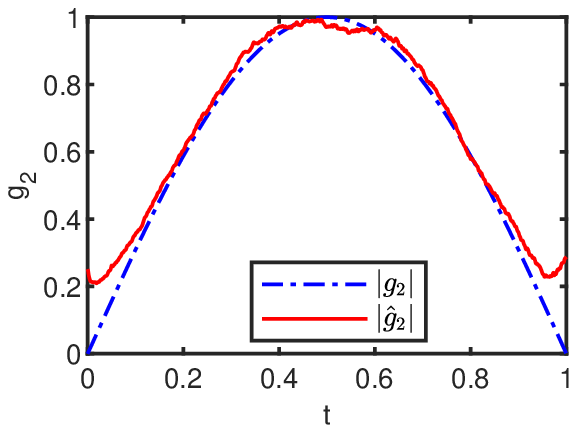}
  %\caption{fig2}
  \end{minipage}%
  }%
  \subfigure[$|\hat g_2|$ when $2.58\sigma =1 $. ]{
  \begin{minipage}[t]{0.33\textwidth}
  \centering
  \includegraphics[width=\textwidth]{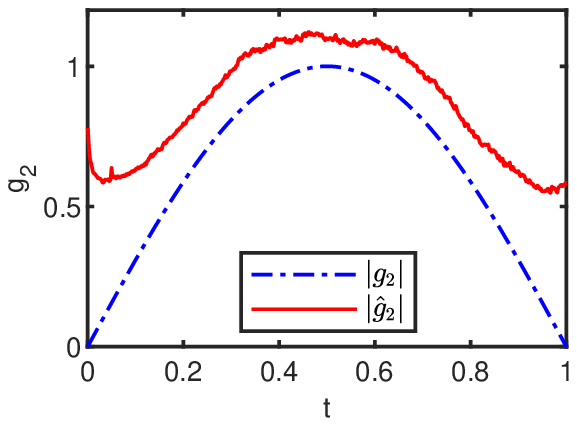}
  %\caption{fig2}
  \end{minipage}%
  }%

  \centering
  \caption{ Experiment $(e3)$, $g_1, \hat g_1$ (top) and 
  $|g_2|, |\hat g_2|$ (bottom) for different $\sigma$.}\label{fig:e3_g&ob}
\end{figure}
\begin{figure}[htbp]
  \centering

  \subfigure[ $er(g_1)$ with various $\sigma$. ]{
  \begin{minipage}[t]{0.5\textwidth}
  \centering
  \includegraphics[width=\textwidth]{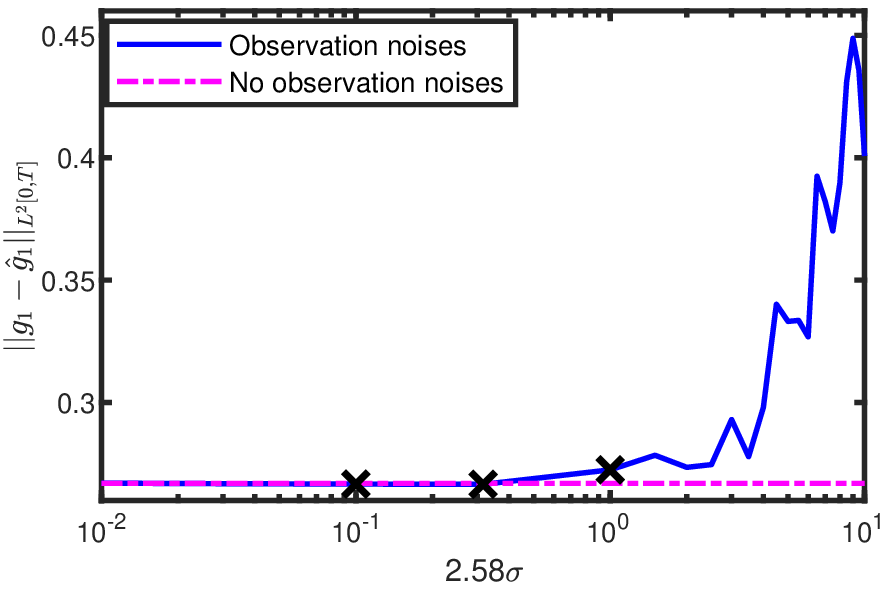}
  %\caption{fig1}
  \end{minipage}%
  }%
  \subfigure[ $er(|g_2|)$ with various $\sigma$. ]{
  \begin{minipage}[t]{0.5\textwidth}
  \centering
  \includegraphics[width=\textwidth]{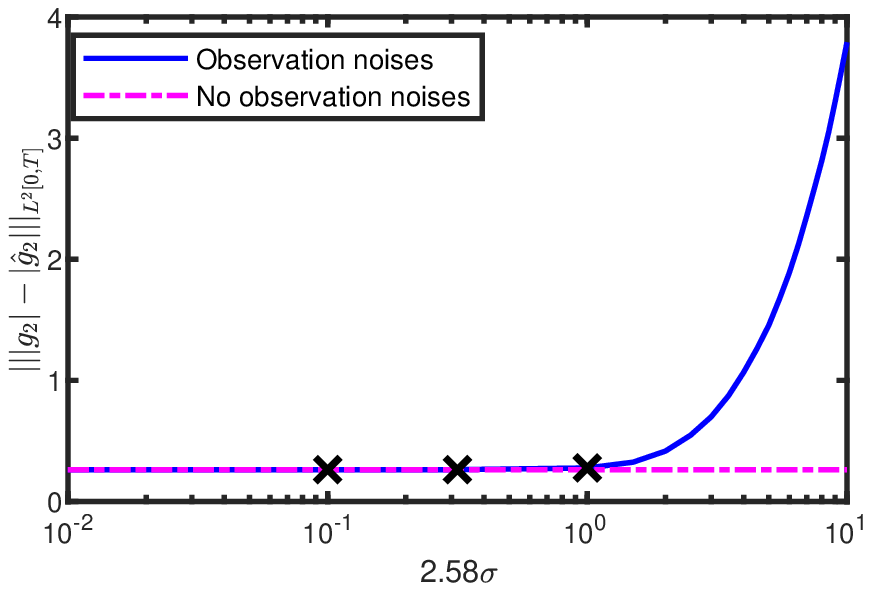}
  %\caption{fig2}
  \end{minipage}%
  }%

  \centering
  \caption{ Experiment $(e4)$, $er(g_1)$ (left) and $er(|g_2|)$ (right) 
  for different $\sigma$.}\label{fig:e4_ob&err}
\end{figure}
\begin{figure}[htbp]
  \centering

  \subfigure[$\hat g_1$ when $2.58\sigma =10^{-1}$. ]{
  \begin{minipage}[t]{0.33\textwidth}
  \centering
  \includegraphics[width=\textwidth]{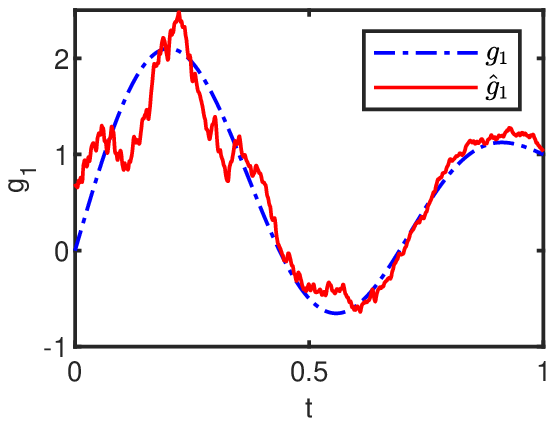}
  %\caption{fig1}
  \end{minipage}%
  }%
  \subfigure[ $\hat g_1$ when $2.58\sigma =10^{-\frac{1}{2}}$.  ]{
  \begin{minipage}[t]{0.33\textwidth}
  \centering
  \includegraphics[width=\textwidth]{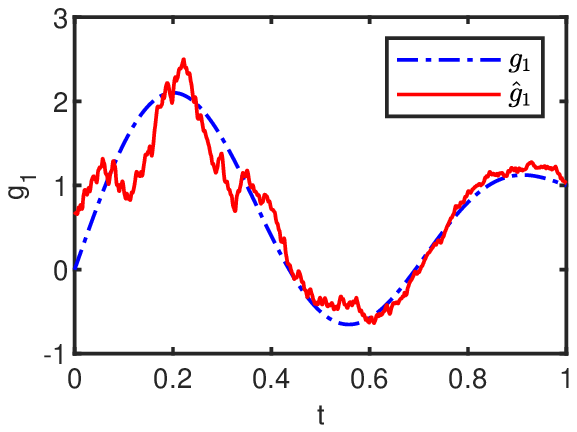}
  %\caption{fig2}
  \end{minipage}%
  }%
  \subfigure[ $\hat g_1$ when $2.58\sigma =1 $.  ]{
  \begin{minipage}[t]{0.33\textwidth}
  \centering
  \includegraphics[width=\textwidth]{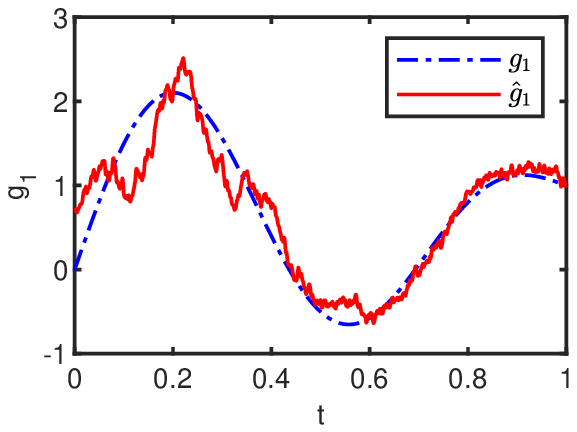}
  %\caption{fig2}
  \end{minipage}%
  }%
  \quad
  \subfigure[$|\hat g_2|$ when $2.58\sigma =10^{-1} $. ]{
  \begin{minipage}[t]{0.33\textwidth}
  \centering
  \includegraphics[width=\textwidth]{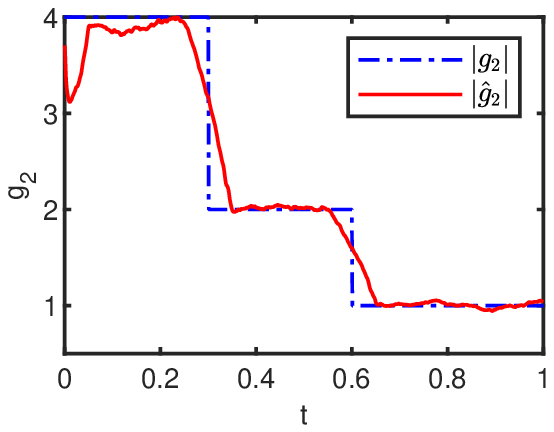}
  %\caption{fig1}
  \end{minipage}%
  }%
  \subfigure[$|\hat g_2|$ when $2.58\sigma = 10^{-\frac{1}{2}} $. ]{
  \begin{minipage}[t]{0.33\textwidth}
  \centering
  \includegraphics[width=\textwidth]{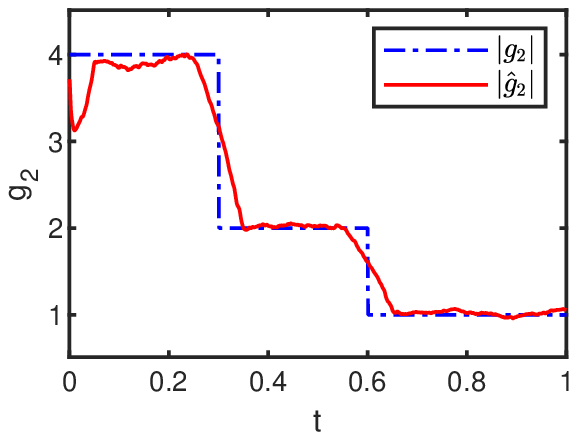}
  %\caption{fig2}
  \end{minipage}%
  }%
  \subfigure[$|\hat g_2|$ when $2.58\sigma =1 $. ]{
  \begin{minipage}[t]{0.33\textwidth}
  \centering
  \includegraphics[width=\textwidth]{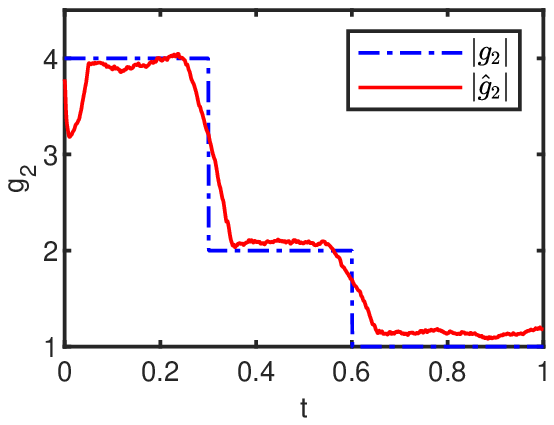}
  %\caption{fig2}
  \end{minipage}%
  }%

  \centering
  \caption{ Experiment $(e4)$, $g_1, \hat g_1$ (top) and 
  $|g_2|, |\hat g_2|$ (bottom) for different $\sigma$.}\label{fig:e4_g&ob}
\end{figure}
\begin{figure}[htbp]
  \centering

  \subfigure[ $er(g_1)$ with various $\sigma$. ]{
  \begin{minipage}[t]{0.5\textwidth}
  \centering
  \includegraphics[width=\textwidth]{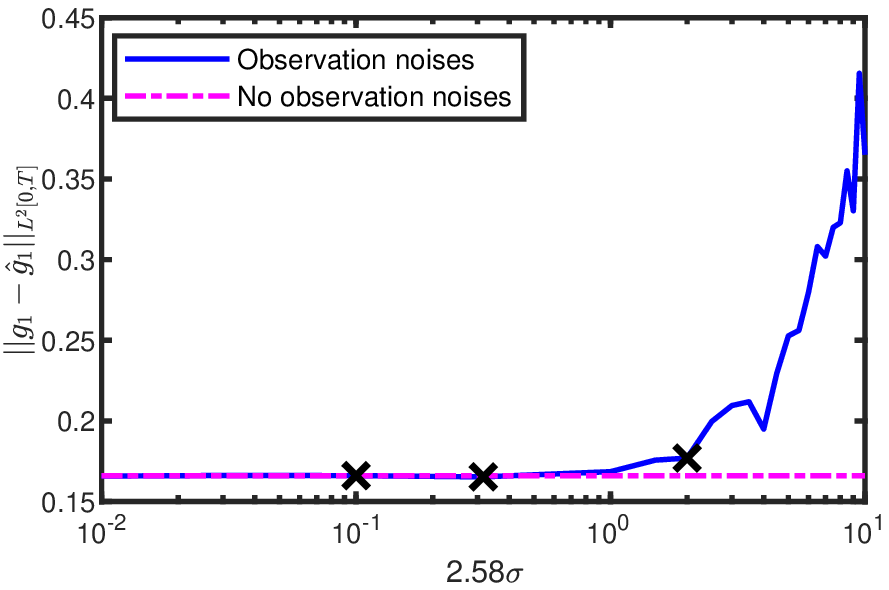}
  %\caption{fig1}
  \end{minipage}%
  }%
  \subfigure[ $er(|g_2|)$ with various $\sigma$. ]{
  \begin{minipage}[t]{0.5\textwidth}
  \centering
  \includegraphics[width=\textwidth]{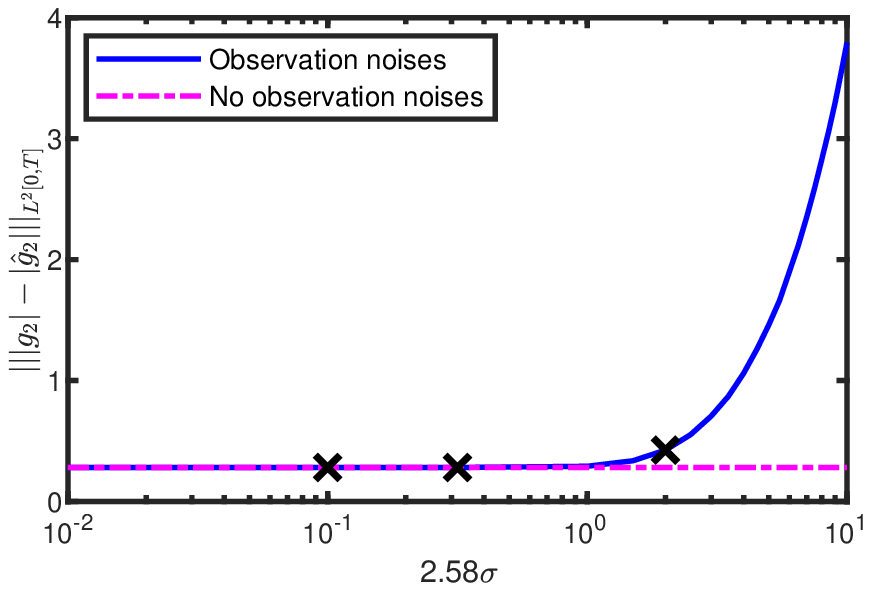}
  %\caption{fig2}
  \end{minipage}%
  }%

  \centering
  \caption{ Experiment $(e5)$, $er(g_1)$ (left) and $er(|g_2|)$ (right) 
  for different $\sigma$.}\label{fig:e5_ob&err}
\end{figure}
\begin{figure}[htbp]
  \centering

  \subfigure[$\hat g_1$ when $2.58\sigma =10^{-1}$. ]{
  \begin{minipage}[t]{0.33\textwidth}
  \centering
  \includegraphics[width=\textwidth]{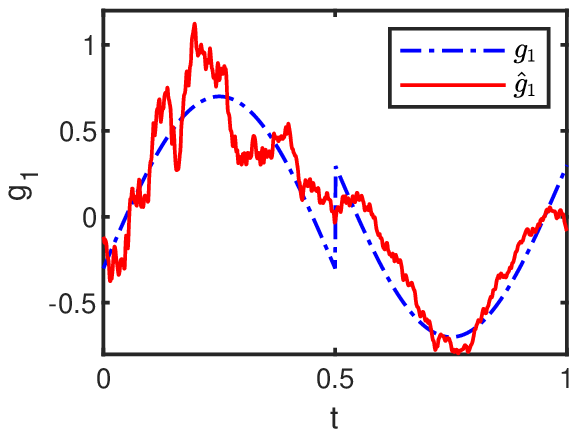}
  %\caption{fig1}
  \end{minipage}%
  }%
  \subfigure[ $\hat g_1$ when $2.58\sigma =10^{-\frac{1}{2}}$.  ]{
  \begin{minipage}[t]{0.33\textwidth}
  \centering
  \includegraphics[width=\textwidth]{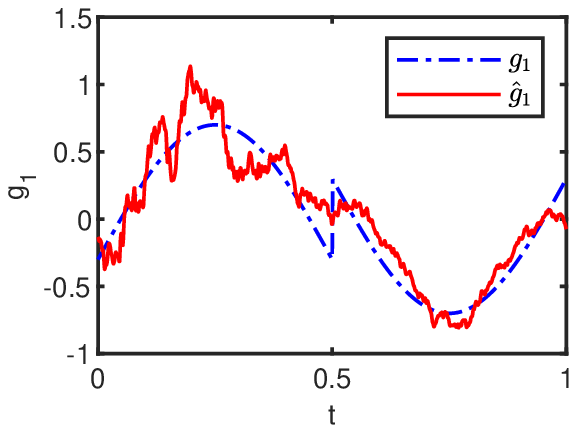}
  %\caption{fig2}
  \end{minipage}%
  }%
  \subfigure[ $\hat g_1$ when $2.58\sigma = 2 $.  ]{
  \begin{minipage}[t]{0.33\textwidth}
  \centering
  \includegraphics[width=\textwidth]{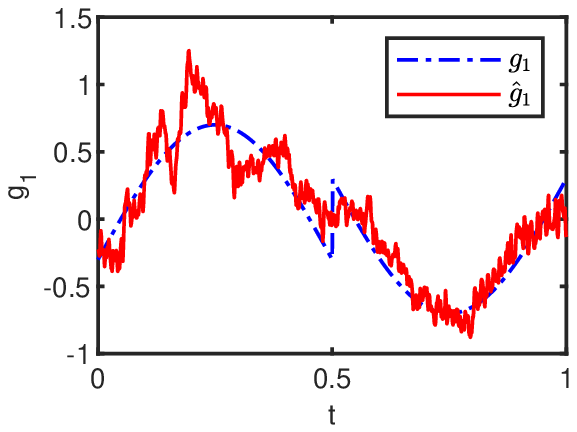}
  %\caption{fig2}
  \end{minipage}%
  }%
  \quad
  \subfigure[$|\hat g_2|$ when $2.58\sigma =10^{-1} $. ]{
  \begin{minipage}[t]{0.33\textwidth}
  \centering
  \includegraphics[width=\textwidth]{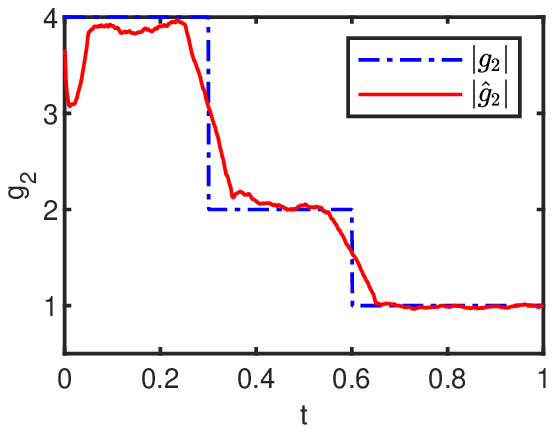}
  %\caption{fig1}
  \end{minipage}%
  }%
  \subfigure[$|\hat g_2|$ when $2.58\sigma = 10^{-\frac{1}{2}} $. ]{
  \begin{minipage}[t]{0.33\textwidth}
  \centering
  \includegraphics[width=\textwidth]{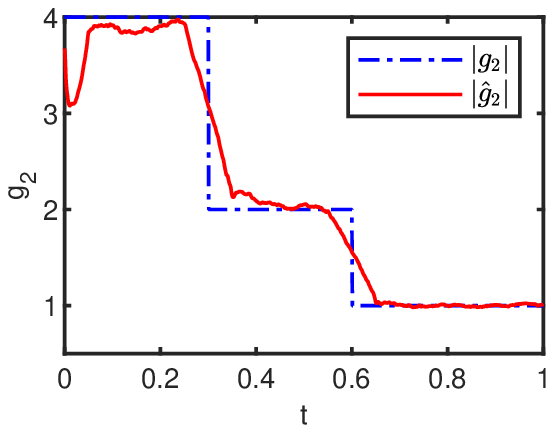}
  %\caption{fig2}
  \end{minipage}%
  }%
  \subfigure[$|\hat g_2|$ when $2.58\sigma = 2 $. ]{
  \begin{minipage}[t]{0.33\textwidth}
  \centering
  \includegraphics[width=\textwidth]{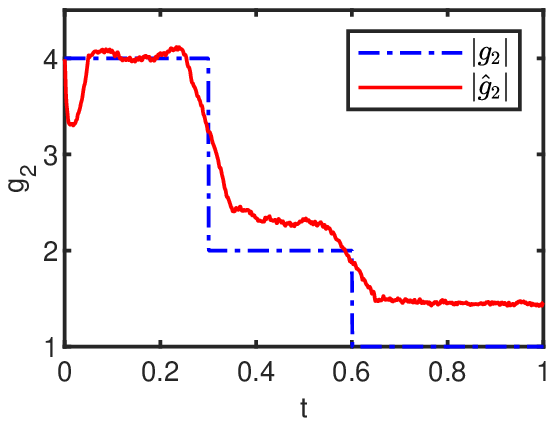}
  %\caption{fig2}
  \end{minipage}%
  }%

  \centering
  \caption{ Experiment $(e5)$, $g_1, \hat g_1$ (top) and 
  $|g_2|, |\hat g_2|$ (bottom) for different $\sigma$.}\label{fig:e5_g&ob}
\end{figure}

\section*{Acknowledgement}
The authors are indebted to Dr. Jin Cheng for his assistance. 
The first author was supported by National Natural Science 
Foundation of China, grant no.11971121, and China Scholarship Council, 
no.201906100094. The third author was supported by Academy of Finland, grants 284715, 
312110 and the Atmospheric mathematics project of University of 
Helsinki.

\bibliographystyle{abbrvurl} % apa abbrv
\bibliography{SFDE_time}

\end{document}